\documentclass[a4paper,11pt]{amsart}
\usepackage{amscd,amsfonts,amssymb}
\usepackage{tikz}
\usepackage{pgfplots}
\usetikzlibrary{automata,positioning,calc,trees}
\usetikzlibrary{intersections,pgfplots.fillbetween}
\usepackage{graphicx,tikz}
\usepackage{pgf,tikz,pgfplots}
\usepackage{mathrsfs}
\usetikzlibrary{arrows}
\usepackage{enumerate}
\usepackage[shortlabels]{enumitem}
\usepackage{mathrsfs}
\usepackage{amssymb,amsmath,amsthm,color}
\usepackage{caption,subcaption}
\usepackage{hyperref}
\usepackage{url}
\usepackage{setspace}
\usepackage{float}

\date{\today}

\allowdisplaybreaks

\newtheorem{theorem}{Theorem}[section]
\newtheorem{proposition}[theorem]{Proposition}
\newtheorem{lemma}[theorem]{Lemma}

\newtheorem{remark}[theorem]{Remark}

\def\be#1 {\begin{equation} \label{#1}}
\newcommand{\ee}{\end{equation}}

\def\sqw{\hbox{\rlap{\leavevmode\raise.3ex\hbox{$\sqcap$}}$%
\sqcup$}}
\def\findem{\ifmmode\sqw\else{\ifhmode\unskip\fi\nobreak\hfil
\penalty50\hskip1em\null\nobreak\hfil\sqw
\parfillskip=0pt\finalhyphendemerits=0\endgraf}\fi}

\newcommand{\R}{{\mathbb {R}}}

\newcommand{\N}{{\mathbb N}}
\newcommand{\Z}{{\mathbb Z}}
\newcommand{\C}{{\mathbb C}}

\newcommand{\K}{\mathcal K}
\newcommand{\m}{\mathfrak m}
\newcommand{\g}{\mathfrak g}

\newcommand{\G}{\mathcal G}
\newcommand{\B}{\mathcal B}

\newcommand{\supp}{\operatorname{supp}}
  \author[Choudhary, Jotsaroop, Shrivastava, Shuin]{Surjeet Singh Choudhary, K. Jotsaroop, Saurabh Shrivastava, Kalachand Shuin}
  \address[Jotsaroop Kaur]{
  	Department of Mathematics\\
  	Indian Institute Science Education and Research\\
  	Mohali, India}
  \email{jotsaroop@iisermohali.ac.in}
  
  \address[Kalachand Shuin]{Department of Mathematical Sciences, Seoul National University, Seoul 08826, Republic of Korea}
  \email{kcshuin21@snu.ac.kr}

\address [Surjeet Singh Choudhary, Saurabh Shrivastava]{
	Department of Mathematics\\
	Indian Institute Science Education and Research Bhopal\\
	Bhopal-462066, India}
\email{surjeet19@iiserb.ac.in, saurabhk@iiserb.ac.in}

\begin{document}
\title[Bilinear Bochner-Riesz square function and applications]{Bilinear Bochner-Riesz square function and applications}
	\begin{abstract} In this paper we introduce Stein's square function associated with bilinear Bochner-Riesz means and investigate its $L^p$ boundedness properties.  Further, we discuss several applications of the square function in the context of bilinear multipliers. In particular, we obtain results for maximal function associated with generalised bilinear Bochner-Riesz means. This extends the results proved in~\cite{JS}. Another application concerns the $L^p$ estimates for bilinear fractional Schr\"{o}dinger multipliers. Finally, we improve upon a result of Grafakos, He and Honzik~\cite{GHH} in the context of bilinear radial multipliers and provide a dimension free sufficient condition on the bilinear multipliers for $L^2\times L^2\rightarrow L^1$ boundedness of the associated maximal function. The generalised bilinear spherical maximal function is a particular example of such maximal functions. 
	\end{abstract}	
	\maketitle
\section{Introduction}\label{sec:intro}

\subsection{Stein's square function}
The square function associated with Bochner-Riesz means was introduced by Stein in~\cite{St}. It is commonly referred to as Stein's square function and is defined by 
\begin{eqnarray*}
	G^{\alpha}(f)(x):=\left(\int_0^{\infty}|\frac{\partial}{\partial t}B_t^{\alpha+1}f(x)|^2 tdt \right)^{\frac{1}{2}}
	= \left(\int_{0}^{\infty}|K^{\alpha}_t\ast f(x)|^2\frac{dt}{t}\right)^{\frac{1}{2}},
\end{eqnarray*}
where $\widehat{B_t^{\alpha}f}(\xi)=\left(1-\frac{|\xi|^2}{t^2}\right)^{\alpha}_+\hat f(\xi)$ is the classical Bochner-Riesz operator with index $\alpha$. Note that the kernel is given by $\widehat{K^{\alpha}_t}(\xi)=2(\alpha+1)\frac{|\xi|^2}{t^2}
\left(1-\frac{|\xi|^2}{t^2}\right)^{\alpha}_+.$ Here $\hat f$ denotes the Fourier transform of $f$ defined by $\hat{f}(\xi)=\int_{\R^n}f(x)e^{-2\pi ix.\xi}dx$. 

The square function naturally appears in the study of maximal Fourier multiplier operators and plays a crucial role. We refer the reader to~\cite{St, S, Car,Car2, Ch, Lee,Leesqr,LRS,LRS2} and references there in for details. 

The $L^p$ estimates 
\begin{eqnarray}\label{squarefun}
	\|G^{\alpha}(f)\|_{p} \lesssim \|f\|_{p}
\end{eqnarray}
for the square function $G^{\alpha}$ have been studied extensively in the literature. The Plancherel theorem yields $L^2(\R^n)$ boundedness of $G^{\alpha}$ for $\alpha>-\frac{1}{2}$, see~\cite{St}. For $p\neq 2$, it is conjectured that the estimate~(\ref{squarefun}) holds for $1<p< 2$  if, and only if $\alpha>n(\frac{1}{p}-\frac{1}{2})-\frac{1}{2}.$ Whereas for the range $p>2$ it is conjectured that the estimate~(\ref{squarefun}) holds if, and only if $\alpha>\alpha(p)-\frac{1}{2},$ where $$\alpha(p)= \max\left\{n\left|\frac{1}{p}-\frac{1}{2}\right|-\frac{1}{2},0\right\}.$$

The conjecture for the range $1<p<2$ has been settled, i.e.   the estimate~(\ref{squarefun}) holds if, and only if $\alpha>n(\frac{1}{p}-\frac{1}{2})-\frac{1}{2}$ for $1<p\leq 2$ and $n\geq 1$. The proof uses the idea of Stein's analytic interpolation for a family of operators between $L^2$ estimate for $\alpha>-\frac{1}{2}$ and $L^p$ estimate for $\alpha>\frac{n-1}{2}$, see~\cite{S,LRS2} for details. Further, in dimensions $n=1,2$, the conjecture has been proved to hold for the range $p> 2$ as well, see~\cite{KanSun} and \cite{Car} for the case of $n=1$ and $n=2$ respectively. However, for $n\geq 3$ and $p>2$ the sufficient part of the conjecture is not known yet completely. There are many interesting developments in this direction, see~\cite{Ch, See, LRS2, LRS, Leesqr} and references therein for more details. In order to state the recent development on the conjecture we set some notation. 

For $n\geq2,$ define $p_0(n)=2+\frac{12}{4n-6-k}$ where $n\equiv k~\text{mod}~3, k=0,1,2$.
Denote $$\mathfrak p_n=\text{min}\left\{p_0(n),\frac{2(n+2)}{n}\right\}.$$
Lee~\cite{Leesqr} proved the following result.
\begin{theorem}\cite{Leesqr}\label{square:Lee}
	For $n\geq 2$, the square function $G^{\alpha}$ satisfies the estimate~(\ref{squarefun}) for $p\geq \min\{\mathfrak p_n,\frac{2(n+2)}{n}\}$ and $\alpha>n(\frac{1}{2}-\frac{1}{p})-1$.
\end{theorem}
Motivated by the recent progress on the bilinear Bochner-Riesz problem and a wide scope of applications of Stein's square function, in this paper we introduce and study the bilinear analogue of the Stein's square function. Consequently, we discuss several connections of the square function in the context of bilinear multipliers. This allows us to obtain new results for maximal function associated with generalised bilinear Bochner-Riesz means and bilinear fractional Schr\"{o}dinger operator. Also, we improve upon a result by Grafakos, He and Honzik~\cite{GHH} for radial bilinear multipliers. These results are described in Sections~\ref{bbr}, \ref{sec:sch} and \ref{sec:brm}. 
Let us first briefly recall some recent developments in the direction of bilinear Bochner-Riesz means. 
\subsection{Bilinear Bochner-Riesz means and maximal function}\label{bbr}
For Schwartz class functions $f,g \in \mathcal{S}(\R^n), n\geq 1$ and $\alpha\geq 0$, the bilinear Bochner-Riesz mean $\mathcal{B}_R^{\alpha}(f,g)$ is defined by 
\begin{eqnarray*}\label{def:bbr}\B^{\alpha}_R(f,g)(x)=\int_{\R^{n}}\int_{\R^{n}}\left(1-\frac{|\xi|^2+|\eta|^2}{R^2}\right)^{\alpha}_{+}\hat{f}(\xi)\hat{g}(\eta)e^{2\pi ix.(\xi+\eta)}d\xi d\eta, ~R>0.
\end{eqnarray*}

The problem of finding necessary and sufficient conditions on exponents $p_1, p_2, p$ and the index $\alpha$ for which the estimate
\begin{eqnarray}\label{esti:bbr}
	\|\B_R^{\alpha}(f,g)\|_p \lesssim \|f\|_{p_1}\|g\|_{p_2}
\end{eqnarray}
holds is commonly referred to as the bilinear Bochner-Riesz problem. Observe that due to standard dilation argument it is enough to study the estimate \eqref{esti:bbr} with $R=1$. For $R=1$ we drop the suffix $R$ from $\B^{\alpha}_R$ and simply denote it by $\B^{\alpha}$.

If $\alpha=0$ the operator $\B^{0}$ (denoted by $\B$) is called the bilinear ball multiplier operator. In dimension $n=1$ Grafakos and Li~\cite{GL} proved the estimate~\eqref{esti:bbr} for $\B$ for all $p_1, p_2, p$ satisfying $2\leq p_1,p_2,p'<\infty$. Here $p'$ denotes the conjugate index of $p$ given by $\frac{1}{p}+\frac{1}{p'}=1$. The range $L=\{(p_1,p_2,p): 2\leq p_1,p_2,p'<\infty\}$ is referred to as the local $L^2-$range of exponents. Later,  Diestel and Grafakos~\cite{DG} showed that in dimension $n\geq 2$ the operator $\B$ fails to satisfy the estimate~\eqref{esti:bbr} provided exactly one of $p_1, p_2$ or  $p'$ is less than $2.$  In~\cite{BGSY} Bernicot et al. initiated the study of the operator $\B^{\alpha}, \alpha>0$ and established the estimate~\eqref{esti:bbr} under certain conditions on $p_1,p_2, \alpha$ and the dimension $n$. In dimension $n=1$ the results proved in~\cite{BGSY} provide an almost complete picture in the Banach triangle, i.e., $1\leq p_1,p_2,p \leq \infty$.  However, in higher dimensions the results are far from being optimal. Liu and Wang~\cite{LW} extended some of these result, specifically to the non-Banach traingle (i.e. when $p<1$) thereby improved the range of $p_1,p_2, p$. Later, Jeong, Lee, and Vargas~\cite{JLV} improved the range of exponents significantly when $p_1, p_2\geq 2$ and lowered the bounds on $\alpha$ for the estimate \eqref{esti:bbr}. They decomposed the operator $\B^{\alpha}$ into discretized square functions and obtained new results, see [Section 3, \cite{JLV}] for details. In particular,  they proved optimal result for the estimate \eqref{esti:bbr} when $p_1=p_2=2$ and $\alpha>0$ for all $n\geq 2$. 

The maximal function associated with the bilinear Bochner-Riesz means, defined by 
$$\B^{\alpha}_*(f,g)(x)=\sup_{R>0} |\B^{\alpha}_R(f,g)(x)|$$
plays a key role in addressing the almost everywhere convergence of the bilinear Bochner-Riesz means $\B_R^{\alpha}(f,g)$ as $R\rightarrow \infty$.  We refer to Grafakos, He and Hon\'{z}ik~\cite{GHH} and Jeong and Lee~\cite{JL} for initiating the study of $L^p$ estimates for $\B^{\alpha}_*$. Recently,  Jotsaroop and  Shrivastava~\cite{JS} introduced a different approach to study $L^p$ boundedness of bilinear maximal function $\B^{\alpha}_*$. Their approach works uniformly in all dimensions. They recovered the results obtained in~\cite{JLV} for Bochner-Riesz means when  $n\geq 2$ and provided new and improved results for the case of dimension $n=1$ for exponents in the non-Banach triangle. We also refer to ~\cite{JSK} for weighted estmates for the bilinear Bochner-Riesz means $\B^{n-\frac{1}{2}}$. 
\subsection{Bilinear Bochner-Riesz square function}
The bilinear Bochner-Riesz square function of order $\alpha$, denoted by $\G^{\alpha}$,  is defined by 
\begin{eqnarray}\label{bilinearsquare}
	\G^{\alpha}(f,g)(x)&:=&\left(\int_0^{\infty}\left|\frac{\partial}{\partial R}\mathcal{B}_R^{\alpha+1}(f,g)(x)\right|^2 RdR \right)^{\frac{1}{2}}.
\end{eqnarray}
Note that $\frac{\partial}{\partial R}\mathcal{B}_R^{\alpha+1}(f,g)(x)$ makes sense for $\alpha>-1$ for each $R>0$.  We rewrite the square function  $\G^{\alpha}(f,g)$ in the following way. 
\begin{eqnarray}
	\G^{\alpha}(f,g)(x)	&=&\nonumber\left(\int_{0}^{\infty}|\mathcal K^{\alpha}_R\ast (f,g)(x)|^2\frac{dR}{R}\right)^{\frac{1}{2}},
\end{eqnarray}
where $\K^{\alpha}_R\ast (f,g)(x)= \K^{\alpha}_R\ast (f\otimes g)(x,x)$. The kernel $\K^{\alpha}_R$ is given by $\widehat{\mathcal K^{\alpha}_R}(\xi,\eta)=2(\alpha+1)\frac{|\xi|^2+|\eta|^2}{R^2}
\left(1-\frac{|\xi|^2+|\eta|^2}{R^2}\right)^{\alpha}_+$. In the spatial variables the kernel is of the form 
$${\mathcal K^{\alpha}_R}(y_1,y_2) = c_{n+\alpha}R^{2n-2}\Delta\left(\frac{ J_{\alpha+n} (|(Ry_1,Ry_2)|)} {|(Ry_1,Ry_2)|^{\alpha+n}}\right),~y_1,y_2\in \R^n.$$
Here $J_{\alpha+n}$ denotes the Bessel function of order $\alpha+n$. 

Let us denote 
\begin{eqnarray*}\label{gfun}
	\g_R^{\alpha}(f,g)(x)=\K^{\alpha}_R\ast (f, g)(x).
\end{eqnarray*}

Our main goal in this paper is to investigate the necessary and sufficient conditions on the exponents $p_1,p_2, p$ and $\alpha$ so that the estimate 
\begin{eqnarray}\label{esti:bbs}
	\|\G^{\alpha}(f,g)\|_p \lesssim \|f\|_{p_1}\|g\|_{p_2}
\end{eqnarray}
holds. Here the notation $A\lesssim B$ in the above means that there is an implicit constant $C>0$ such that $A\leq CB.$ The constant $C$ is independent of essential quantities like functions appearing in the estimate. However, it may depend on parameters $\alpha, n,p_1$ and $p_2$. Sometimes, we will also use the notation $A\lesssim_{\epsilon} B$ to emphasize the dependence of the implied constant on the parameter $\epsilon$. 
%
Also, we use calligraphy letter to denote the bilinear operators whereas the corresponding capital letters are used to denote the operators from the theory of linear operators.
\subsection*{Organization of the paper} In Section~\ref{sec:main} we provide statements of  results of this paper. The proofs of square function boundedness results, namely Theorems~\ref{maintheorem:sqr} and \ref{dim1} are given in Sections~\ref{sec:pqr} and \ref{sec:dim1} respectively. 
We discuss the idea of analytic interpolation for $\mathcal G^{\alpha}$ in Section~\ref{sec:inter} and prove Theorem~\ref{thm:inter}. The necessary conditions on $\alpha$ are obtained in Section~\ref{sec:nec}. Finally, the sparse domination result Theorem~\ref{thm:critical} is proved in the Appendix~\ref{sec:critical}. 
 

\section{Results}\label{sec:main}
\subsection{$L^p$ estimates for bilinear square function $\G^{\alpha}$ }
Let $n\geq 2$ and $1\leq p_1,p_2\leq \infty$. Let us consider the following notation. 
\begin{equation*}
	\alpha_*(p_1,p_2)=\begin{cases}\alpha(p_1)+\alpha(p_2)&\textup{when}~~ \mathfrak p_n\leq p_1,p_2\leq\infty;\\\\
		\alpha(p_1)+\left(\frac{1-2p_2^{-1}}{1-2(\mathfrak p_n)^{-1}}\right)\alpha(\mathfrak p_n)&	\textup{when}~~ \mathfrak p_n\leq p_1\leq\infty~ \text{and}~2\leq p_2<\mathfrak p_n;\\\\
		\left(\frac{1-2p_1^{-1}}{1-2(\mathfrak p_n)^{-1}}\right) \alpha(\mathfrak p_n)+\alpha(p_2)&	\textup{when}~~ 2\leq p_1<\mathfrak p_n~ \text{and}~\mathfrak p_n\leq p_2\leq\infty;\\\\
		\left(\frac{2-2p_1^{-1}-2p_2^{-1}}{1-2(\mathfrak p_n)^{-1}}\right)\alpha(\mathfrak p_n)&\textup{when}~~2\leq p_1,p_2 <\mathfrak p_n.
	\end{cases}
\end{equation*}
The following $L^p$-boundedness results for $\mathcal G^{\alpha}$ hold.  
\begin{theorem}\label{dim1} Let $n=1$ and $1<p_1,p_2<\infty$ be such that $\frac{1}{p_1}+\frac{1}{p_2}=\frac{1}{p}$. The bilinear Bochner-Riesz square function $\mathcal G^{\alpha}$ maps $L^{p_1}(\R)\times L^{p_2}(\R)$ into $ L^{p}(\R)$ for each of the following cases. 
	\begin{enumerate}
		\item $p_1,p_2\geq 2$ and $\alpha>0$.
		\item $1<p_1<2, p_2\geq 2$ and $\alpha>\frac{1}{p_1}-\frac{1}{2}$.
		\item $1<p_2<2, p_1\geq 2$ and $\alpha>\frac{1}{p_2}-\frac{1}{2}$.
		\item \label{case4} $1<p_1,p_2<2$ and $\alpha>\frac{1}{p}-1$. 
	\end{enumerate}
\end{theorem}
\begin{theorem} \label{maintheorem:sqr} Let $n\geq 2$ 
	and $(p_1,p_2,p)$ be such that $p_1,p_2\geq 2$ and $\frac{1}{p}=\frac{1}{p_1}+\frac{1}{p_2},$ then for $\alpha>\alpha_*(p_1,p_2)$ the bilinear Bochner-Riesz square function $\mathcal G^{\alpha}$ maps $L^{p_1}(\R^n)\times L^{p_2}(\R^n)$ into $ L^{p}(\R^n)$. 
\end{theorem}

Next, we make use of Stein's interpolation for analytic family of bilinear operators (see [Theorem 7.2.9,  \cite{Grafakosmodern}]) to extend boundedness of $\mathcal G^{\alpha}$ when either of the exponents $p_1$ or $p_2$ is less than $2$.  This idea requires $L^p$ estimates for $\mathcal G^{\alpha}$ when $\alpha$ is larger than the critical index $n-\frac{1}{2}$. The $L^p$ estimates for $\mathcal G^{\alpha}$ for $\alpha>n-\frac{1}{2}$ can be easily proved using the arguments from its linear counterpart. Since this part does not require any non-trivial modification in the existing arguments, we skip the details for now and  provide them in the Appendix for completeness. 
\begin{theorem}\label{thm:inter}
Let $1<p<2$, then $\G^{\alpha}$ is bounded from $L^{p}(\mathbb{R}^{n})\times L^{p}(\mathbb{R}^{n})$ to $L^{p/2}(\mathbb{R}^{n})$ for $\alpha>(2n-1)(\frac{1}{p}-\frac{1}{2})$.
\end{theorem}
\begin{remark}
	Observe that when $n=1$ we have that $\frac{1}{p}-\frac{1}{2}<\frac{2}{p}-1$ for $p<2$. Therefore, we get an improved range of exponents in Theorem~\ref{thm:inter} as compared to case $(4)$ when $p_1=p_2$ in Theorem~\ref{dim1}. 
\end{remark}
Next result describes necessary conditions for the $L^p$ boundedness of the square function $\G^{\alpha}$.
\begin{proposition}\label{prop:neccond}
	Assume that $\mathcal{G}^{\alpha}$ is bounded from $L^{p_1}(\R^n)\times L^{p_2}(\R^n)\rightarrow L^p(\R^n)$. Then the exponents satisfy the following necessary conditions. 
	\begin{enumerate}
		\item $\alpha>\max\left\{n\left(\frac{1}{p}-1\right)-\frac{1}{2},-\frac{1}{2}\right\}$ for all $n\geq 1$ and $p_1,p_2\geq 1.$
		\item $\alpha>\max\left\{\frac{n}{2}-\frac{n}{p_1}-\frac{n}{2p_2}-1,\,\frac{n}{2}-\frac{n}{2p_1}-\frac{n}{p_2}-1,- \frac{1}{2}\right\}$ where $n\geq 2$ and $1\leq p_1,p_2,p \leq \infty$. 
	\end{enumerate} 	
\end{proposition}

Finally, we show that  $L^p$-estimates for $\G^{\alpha}$ can be used to prove new results for various types of bilinear operators. 
\subsection{Generalised Bilinear Bochner-Riesz means and maximal function}\label{bbr}
Let $\alpha, \lambda>0$. For Schwartz class functions $f,g \in \mathcal{S}(\R^n), n\geq 1$ consider  the generalised bilinear Bochner-Riesz mean $\mathcal{B}_{\lambda,R}^{\alpha}(f,g)$ defined by 
\begin{eqnarray*}\label{def:bbr}\B^{\alpha}_{\lambda, R}(f,g)(x)=\int_{\R^{n}}\int_{\R^{n}}\left(1-\frac{|(\xi,\eta)|^{\lambda}}{R^{\lambda}}\right)^{\alpha}_{+}\hat{f}(\xi)\hat{g}(\eta)e^{2\pi ix.(\xi+\eta)}d\xi d\eta, ~R>0,
\end{eqnarray*}
Note that when $\lambda=2$ we have $\B^{\alpha}_{\lambda, R}(f,g)(x)=\B^{\alpha}_{R}(f,g)(x)$. Consider the maximal function 
$$\B^{\alpha}_{\lambda, *}(f,g)(x)=\sup_{R>0}|\B^{\alpha}_{\lambda, R}(f,g)(x)|.$$
The following estimate holds.  
\begin{theorem}\label{GBBR} 
	Let $n\geq 2$ and $\alpha>\alpha_*(p_1,p_2)+1/2$. Then the maximal function $\mathcal{B}^{\alpha}_{*,\lambda}(f,g)$ maps $L^{p_1}(\R^{n})\times L^{p_2}(\R^n)\rightarrow L^p(\R^n)$ where $ p_1,p_2\geq 2$ and $\frac{1}{p}=\frac{1}{p_1}+\frac{1}{p_2}$.	
\end{theorem}
 Note that invoking $L^p$ boundedness results for the square function from Theorem~\ref{dim1} and ~\ref{maintheorem:sqr} we get $L^p$ estimates for the maximal function $\mathcal{B}^{\alpha}_{\lambda,*}(f,g)(x)$. This generalises the results proved in~\cite{JLV,JS} to the setting of generalised bilinear Bochner-Riesz means. We would like to emphasise here that the methods used in~~\cite{JLV,JS} explicitly use the fact $\lambda=2$ and do not apply to the case of $\lambda\neq 2$ directly. Therefore, the results obtained in Theorem~\ref{GBBR} are new for $\lambda\neq 2.$  This is possible due to the use of square function. However we conjecture that the range of $L^p$ boundedness of $\B^{\alpha}_{\lambda, *}$ for $\lambda\neq 2$ in theorem \ref{GBBR} above should be similar to the standard maximal bilinear Bochner Riesz means, i.e. the case of $\lambda=2$.

\subsection{Bilinear fractional Schr\"{o}dinger multiplier} \label{sec:sch}
The fractional Schr\"{o}dinger equation is defined as
\begin{eqnarray*}
	\frac{\partial}{\partial s}u(s,x)=\left(-\triangle\right)^{\beta}u(s,x)&\text{when}~s>0\\
	u(0,x) =f(x).
\end{eqnarray*}
The solution to this equation with initial data $f$ is of the form 
$$e^{is\left(-\triangle\right)^{\beta}} f(x)=\int_{\R^n}m_{\beta}(s|\xi|^2)\widehat{f}(\xi)e^{2\pi i x\cdot \xi} d\xi, s>0,$$ 
where $m_{\beta}(u)= e^{i|u|^{\beta}}, u\in \R,\beta>0$ and it is called the fractional Schr\"{o}dinger multiplier. The problem of finding optimal $\gamma$ for $\beta=1$ such that $e^{-is\triangle}f\rightarrow f$ a.e. as $s\rightarrow 0$ when $(I+ \left(-\triangle\right)^{\gamma/2})f\in L^2(\R^n) $ has been resolved recently in {\cite{DGL, DZ}}. 
When $\gamma >\frac{n}{2(n+1)}$ and $(I+ \left(-\triangle\right)^{\gamma/2})f\in L^2(\R^n) $ it is known that $\lim_{s\rightarrow 0^+}e^{-is\triangle}f= f$ a.e. This result is sharp except at the end-point $\gamma=\frac{n}{2(n+1)}$ for $n\geq 2$, see {\cite{DGL, DZ}}. When $n=1$ it was proved that $e^{-is\triangle}f\rightarrow f$ a.e. as $s\rightarrow 0^+$ if and only if $\gamma\geq \frac{1}{4}$, see \cite{DK2,C}.

We consider the bilinear fractional Schr\"{o}dinger multiplier operator defined by   \begin{equation}\label{fracshrodinger}
	\mathcal T_{m_\beta,s}(f,g)(x)=\int_{\R^n}\int_{\R^n}m_{\beta}(s^{2}|(\xi,\eta)|^2)\hat{f}(\xi)\hat{g}(\eta)e^{2\pi i x\cdot(\xi+\eta)} d\xi d\eta.\end{equation}
Note that  $\mathcal T_{m_\beta,s}(f,g)$ solves the fractional Schr\"{o}dinger equation when the initial data is $f(\cdot)g(\cdot)$. We are concerned with the problem of convergence of 
$\mathcal T_{m_\beta,s}(f,g)$ a.e. to $f(\cdot)g(\cdot)$ when $(I+ \left(-\triangle\right)^{\gamma/2})f,(I+ \left(-\triangle\right)^{\gamma/2})g \in L^2(\R^n)$ for some $\gamma$.  In order to address this problem we   establish the following result for the associated maximal function.
 \begin{theorem}\label{schrodinger}
	Let $\mathcal T_{m_\beta,s}(f,g)$ be as defined in \eqref{fracshrodinger}. If $\beta\left(\alpha_*(p_1,p_2)+1\right)<\gamma$ then \begin{equation}\label{fracshrodinger2}\|\sup_{0<s<1}|\mathcal T_{m_\beta,s}(f,g)|\|_{p}\lesssim \|(I -\triangle)^{\gamma}f\|_{p_1} \|(I -\triangle)^{\gamma}g\|_{p_2}.\end{equation}	
	Consequently, we get that $\mathcal T_{m_\beta,s}(f,g)(x)\rightarrow f(x) g(x)$ as $s\rightarrow 0$ for a.e. $x$ whenever the right hand side of ~\eqref{fracshrodinger2} is finite.
\end{theorem}
\begin{remark} Since $\alpha_*(2,2)=0$, Theorem \ref{schrodinger} implies that $\mathcal T_{m_\beta,s}(f,g)(x)\rightarrow f(x)g(x)$ a.e. for any $\gamma>\beta$ provided $(I -\triangle)^{\gamma}f,(I -\triangle)^{\gamma}g\in L^2(\R^n)$. Note that when $\beta=1$ we have $m_{1}\left(|(\xi,\eta)|^2\right)=e^{i|\xi|^2}e^{i|\eta|^2}$. In this case we can directly use the result from the linear theory to prove the estimate  \eqref{fracshrodinger2} for $\gamma>\frac{n}{2(n+1)}$. When $\beta\neq 1$ the results of the type \eqref{fracshrodinger2} are new in the bilinear setting and Theorem~\ref{schrodinger} provides us with a range of $\beta$ for which the results holds. However, the problem of finding an optimal regularity for $f,g$ for which \eqref{fracshrodinger2} holds needs to be investigated further.  
\end{remark}
\subsection{General bilinear radial multipliers}\label{sec:brm}
Let $m:\R^{2n}\rightarrow \C$ be a bounded measurable function. 
Let $\mathcal T_{m,s}, s>0$ be the corresponding bilinear multiplier operator defined as $$\mathcal T_{m,s}(f,g)(x)=\int_{\R^{n}}\int_{\R^{n}} m(s\xi,s\eta)\hat{f}(\xi)\hat{g}(\eta)e^{2\pi i x\cdot(\xi+\eta)}d\xi d\eta.$$ 
In the study of bilinear multiplier operators the lack of Plancherel theorem poses a big difficulty. In \cite{GHH}, Grafakos, Honzik and He obtained some sufficient conditions on $m$ so that the corresponding bilinear maximal function  $\mathcal T^*_m(f,g):=\sup_{s>0}|\mathcal T_{m,s}(f,g)|$ is bounded from $L^2(\R^n)\times L^2(\R^n)\rightarrow L^1(\R^n)$. To be precise, they proved that if $m\in C^{\infty}(\R^{2n})$ and satisfies 
\begin{eqnarray*}
	|\partial^{\beta}m(\xi,\eta)|\leq C_{\beta}|(\xi,\eta)|^{-a},~~\forall~~ |\beta|\leq [\frac{n}{2}]+2,
\end{eqnarray*}
where $[\frac{n}{2}]$ is the integer part of $\frac{n}{2}$ and $a>\frac{n}{2}+1,$ then $\mathcal T^*_m$ is a bounded operator from $L^2(\R^n)\times L^2(\R^n)\rightarrow L^1(\R^n)$. Note that the condition as above is dependent on the dimension. Here we provide an improved sufficient condition for the case of bilinear radial multipliers. In doing so we make use of the bilinear Bochenr-Riesz square function. 

Let $m_0:[0,\infty)\rightarrow \C$ be a bounded measurable function and consider $m(\xi,\eta)=m_0(|(\xi,\eta)|^2)$ a radial function on $\R^{2n}$. Let $\mathcal T^*_m(f,g)$ denote the bilinear maximal function associated with $m(\xi,\eta)=m_0(|(\xi,\eta)|^2)$ defined as above. Observe that if $m_0$ is a smooth function on $[0,\infty)$, it is easy to see that $m(\xi,\eta):=m_0(|(\xi,\eta)|^2)$ is also a smooth function on $\R^{2n}$. 
Let $\varphi,\phi:(0,\infty)\rightarrow \C$ be compactly supported smooth functions such that  $\sum_{j\geq 1}\varphi(2^{-j}x)+\phi(x)\equiv 1$ on $[0,\infty)$ and $\supp (\varphi)\subset [1/2,2]$ and $\supp (\phi)\subset[0,3/2]$. We establish the following result concerning sufficient condition on $m_0$ so that $\mathcal T^*_m$ is bounded from $L^2(\R^n)\times L^2(\R^n)\rightarrow L^1(\R^n)$. 
\begin{theorem}\label{bilinearradial}Let $m_0:[0,\infty)\rightarrow \C$ be a smooth function on $[0,\infty)$. Let $$m_{j}(t):= m_0(t)\varphi\left(2^{-j}t\right), t\geq 0$$ and there exists $\epsilon>0$ and $\beta>1$ such that $$\|m_j\|_{L^2_{\beta}}\leq C 2^{-j\epsilon}$$ for all $j\geq 1$ with $C$ independent of $j$.  In particular, the operator $\mathcal T^*_{m}$ extends as a bounded operator from $L^2(\R^n)\times L^2(\R^n)\rightarrow L^1(\R^n)$.
\end{theorem}
 Indeed, the idea of the proof of Theorem ~\ref{bilinearradial} yields the following result in terms of the derivatives of $m_0$.  
\begin{theorem}\label{classicalmultiplier}
	Let $m_0:[0,\infty)\rightarrow \C$ be a smooth function such that 
	$$|m_0^k(t)|\lesssim t^{-k-\epsilon}~~\text{as}~ t\rightarrow \infty$$ 
	for some $\epsilon>0$ and $k=0,1,2$. Here $m_0^k$ denotes the $k$th derivative of $m_0$ Then $\mathcal T_m^*$ is bounded from $L^2(\R^n)\times L^2(\R^n)\rightarrow L^1(\R^n)$.	
\end{theorem}
Theorem~\ref{bilinearradial} can be applied to deduce $L^p$ estimates for the generalized bilinear spherical maximal function, which is defined as follows. \\

For $n\geq 1$ define the generalized bilinear spherical means by 
$$\mathcal S_{\omega_{\mu},s}(f,g)(x)=\int_{\R^n}\int_{\R^n}\omega_{\mu}(s^2|(\xi,\eta)|^2)\widehat{f}(\xi)\widehat{g}(\eta)e^{2\pi ix.(\xi+\eta)} d\xi d\eta,$$
where 
$$\omega_{\mu}(t)= 2^{\mu+n-1}\Gamma(\mu+n+1)\frac{J_{\mu+n}(|t|
	^{\frac{1}{2}})}{|t|^{\frac{1}{2}(\mu +n)}}$$
for $\mu\in\C$ such that $\mu\neq  -n-1, -n-2,...-n-k,...$. Here $J_{\mu}$ denotes the Bessel function of first kind of order $\mu$. 

The generalised bilinear spherical maximal function is defined by 
$$\mathcal S^*_{\omega_{\mu}}(f,g)(x)=\sup_{s>0}|\mathcal S_{\omega_{\mu},s}(f,g)(x)|.$$

Note that when $\beta=-1$ the operator $\mathcal S^*_{\omega_{\beta}}$ is the bilinear spherical maximal function 
$$\mathcal S^*_{\omega_{-1}}(f,g)(x)=\mathcal M_S(f,g)(x)=\sup_{s>0}|\int_{S^{2n-1}}f(x-sy)g(x-sz) d\sigma(y,z)|,$$
where $d\sigma$ is the normalised Lebesgue measure on the unit sphere $S^{2n-1}\subset \R^{2n}.$ We refer to ~\cite{Christ,JL} for results on the bilinear spherical maximal function. 

As a consequence of Theorem~\ref{bilinearradial} we get the following result. 
\begin{theorem}\label{generalized}
	For $\mu> -n+\frac{1}{2}$, we have 
	\begin{equation*}\label{bspher}
		\|\mathcal S^*_{\omega_{\mu}}(f,g)\|_1\lesssim \|f\|_{2}\|g\|_{2}\end{equation*}
\end{theorem}
Observe that the theorem above also includes the bilinear spherical maximal function for $n\geq 2$. 
We will skip the proof of the theorem above. It may be completed using the asymptotic expansion of the Bessel functions along with Theorem~\ref{bilinearradial}. 

\section{Proof of Theorems~\ref{GBBR} and \ref{schrodinger}}\label{appl}
The methods developed in \cite{JL} and \cite{JS} for studying the boundedness of the maximal bilinear Bochner-Riesz means (i.e. when $\lambda=2$) do not apply directly to deduce the $L^p$ estimates for $\mathcal{B}_{\lambda,*}^{\alpha}$ for $\lambda\neq 2$. Therefore, the role of  bilinear square function is crucial here. 
We will establish a pointwise relation between the maximal function $\mathcal{B}_{\lambda,*}^{\alpha}$  and square function $\G^{\alpha}$. In particular, we show that 
 \begin{eqnarray}\label{GBBR2}\mathcal{B}^{\alpha}_{\lambda,*}(f,g)(x)\lesssim_{\lambda,\alpha,\beta}\mathcal{G}^{\beta-1}(f,g)(x) + Mf(x)Mg(x)~~\text{ for~ a.~e. ~} ~~x\in \R^n,
 	\end{eqnarray}
 where $\alpha-\beta+1/2 >0, \beta>1/2$ and $\lambda>0$. Observe that invoking theorem \ref{maintheorem:sqr} regarding the boundedness of bilinear square function, the inequality above yields the desired estimate for $\mathcal{B}^{\alpha}_{\lambda,*}$ in theorem \ref{GBBR}. In order to prove the inequality \eqref{GBBR2} we require the Riemann-Liouville formula (see~\cite{Car2} for details). 

\begin{lemma}\label{L}\cite{Car2} Let $h\in L^2(\R)$ and for $\beta \geq 0$ let $\widehat{\left(d/dt\right)^{\beta}h}(\nu)= (-2\pi i \nu)^{\beta}\hat{h}(\nu)$. Suppose that $\text{supp}(h)\subseteq (-\infty,a] $ and $\left(d/dt\right)^{\beta}h \in L^2(\R)$ for $\beta >\frac{1}{2}$. Then $\text{supp}(\left(d/dt\right)^{\beta}h) \subseteq (-\infty,a] $ and \begin{equation}\label{RLF}h(x)= c_{\beta} \int_x^{\infty}(t-x)^{\beta-1}\left(d/dt\right)^{\beta}h (t)dt,~~\text{for~a.e.~} x.\end{equation} 	
\end{lemma}
Let $m\in C_c^{\infty}((0,\infty))$. Applying Lemma~\ref{L} to  $M(u)=\frac{m(u)}{u}$ we can write 
$$m(u)=c_{\beta}\int_0^{\infty}\left(1-\frac{u}{t}\right)_+^{\beta-1}\frac{u}{t} t^{\beta}\left(d/dt\right)^{\beta}M (t)dt,$$
where $\beta >\frac{1}{2}$. Write $u=|(\xi,\eta)|^2$ in the formula above to get 
\begin{equation}\label{RLI}m(|(\xi,\eta)|^2)=c_{\beta}\int_0^{\infty}\left(1-\frac{|(\xi,\eta)|^2}{t}\right)_+^{\beta-1}\frac{|(\xi,\eta)|^2}{t} t^{\beta}\left(d/dt\right)^{\beta}M (t)dt.\end{equation}

Consider the bilinear multiplier operator given by  $$\mathcal T_{m,s}(f,g)(x)=\int_{\R^n}\int_{\R^n}m(s^{2}|(\xi,\eta)|^2)\hat{f}(\xi)\hat{g}(\eta)e^{2\pi i x.(\xi+\eta)} d\xi d\eta.$$
Observe that using the formula \eqref{RLI} we can rewrite $\mathcal T_{m,s}(f,g)$ as 
\begin{equation}\label{RLI2}
	\mathcal T_{m,s}(f,g)(x)=c_{\beta}\int_0^{\infty}\mathcal{K}^{\beta-1}_{\sqrt{t}/s}*(f,g)(x)t^{\beta}\left(d/dt\right)^{\beta}M (t)dt.
\end{equation}
Cauchy Schwarz inequality in the above yields 
\begin{equation}\label{app1}
	\sup_{s>0}\left|\mathcal T_{m,s}(f,g)(x)\right|\leq c_{\beta} \|m\|_{L^2_{\beta}}~ \mathcal{G}^{\beta-1}(f,g)(x),	
\end{equation}
where $\|m\|_{L^2_{\beta}}^2=\int_0^{\infty}|t^{\beta +1}\left(d/dt\right)^{\beta}M (t)|^2 t^{-1} dt$ and $\left(d/dt\right)^{\beta}M (t)$ is a distributional derivative of $M$ of order $\beta$. 

We make use of the analysis as above to prove Theorem~\ref{GBBR}. 

\subsection*{Proof of theorem \ref{GBBR}}
For convenience let us work with $2\lambda$ in place of $\lambda$. Let $m_j(|(\xi,\eta)|^2)=\psi\left(2^j\left(1-|(\xi,\eta)|^{2\lambda}\right)\right)$, where $\psi$ is a smooth compactly supported function on $[\frac{1}{2},2]$. We can write 
$$\mathcal{B}_{2\lambda,R}^{\alpha}(f,g)(x)=\mathcal{B}_{0,R}(f,g)(x)+\sum_{j\geq 1}2^{-j\alpha}\int_{\R^{2n}}m_j(R^{-2}|(\xi,\eta)|^2)\hat{f}(\xi)\hat{g}(\eta)e^{2\pi ix.(\xi+\eta)} d\xi d\eta,$$ where 
 $$\mathcal{B}_{0,R}(f,g)(x)=\int_{\R^{2n}}\left(1-\frac{|(\xi,\eta)|^{2\lambda}}{R^{2\lambda}}\right)_+^{\alpha}\varphi\left(\frac{|(\xi,\eta)|^{2\lambda}}{R^{2\lambda}}\right)\hat{f}(\xi)\hat{g}(\eta)e^{2\pi ix.(\xi+\eta)}d\xi d\eta$$ and $\varphi(x)+\sum_{j\geq 1}\psi\left(2^j\left(1-x\right)\right)\equiv 1$ on $[0,1)$.
 
 Let $\mathcal{B}_j(f,g)(x)=\sup_{R>0}\left|\int_{\R^{2n}}m_j(R^{-2}|(\xi,\eta)|^2)\hat{f}(\xi)\hat{g}(\eta)e^{2\pi ix.(\xi+\eta)} d\xi d\eta\right|$. For $j\geq 1$ we claim that 
 $$\mathcal{B}_j(f,g)(x)\lesssim 2^{j\beta-j/2}~ \mathcal{G}^{\beta-1}(f,g)(x)~~\text{
 for~ any}~\beta>1/2.$$ 
This follows from the inequality \eqref{app1} along with the estimate $\|m_j\|_{L^2_{\beta}} \simeq 2^{j\beta-j/2}$.
 For, observe that when $\beta=0$ we have that $\|m_j\|_{L^2_{0}}^2=\int_0^{\infty}|\psi(2^j(1-t^{2\lambda}))|^2 t^{-1} dt\simeq 2^{-j}$. Further, when $\beta=1$ we can easily verify that $\|m_j\|_{L^2_{1}}^2\simeq 2^{j}$. Therefore, interpolating between $\beta=0$ and $\beta=1$ we get the estimate for $0<\beta<1$. In fact, the same argument yields the desired bound for any $\beta>0$.
 
 Next, we claim that  $$\sup_{R>0}|\mathcal{B}_{0,R}(f,g)(x)|\lesssim {M}f(x){M}g(x)~~\text{for~every}~ \lambda>0.$$ 
 
Note that $\varphi\equiv 1$ in $[0,\delta]$ for some $\delta>0$ and $\text{supp} (\varphi)\subset[0,1/2]$. 
Let $\rho$ be a smooth function on $\R^d$ supported on $\xi:1/2\leq |\xi|\leq 2 $ such that $\sum_{j\geq 0}\rho(2^j \xi)\equiv 1$ on $\xi:0<|\xi\leq 1|$.
Consider the kernel  
$$K(x)=\int_{\R^d}\left(1-|\xi|^{2\lambda}\right)_+^{\alpha}\varphi(|\xi|^{2\lambda})e^{2\pi ix.\xi}d\xi.$$ 
Let $\rho\in C^{\infty}\left([1/2,1]\right)$ and write 
 $$\tilde{K}(x):=\sum_{j\geq 0}\int_{\R^d}\left((1-|\xi|^{2\lambda})_+^{\alpha}-1\right)\rho(2^{j}\xi)\varphi(|\xi|^{2\lambda})e^{2\pi ix.\xi}d\xi=\sum_{j\geq 0}\tilde{K}_j(x).$$
 We need to estimate $\tilde{K}_j$ for large $j$. When $j$ is large we have that  $\rho(2^{j}\xi)\varphi(|\xi|^{2\lambda})=\rho(2^{j}\xi)$.
 
 Using power series expansion of $(1-t)^{\alpha}, 0\leq t\leq \delta<1$ around $0$ we can write $$(1-|\xi|^{2\lambda})_+^{\alpha}-1=\sum_{k\geq 1}\frac{(-1)^k}{k!}\alpha (\alpha-1)..(\alpha-k+1)|\xi|^{2\lambda k}.$$ It is easy to verify that the integration by parts argument gives us $$|\tilde{K}_j(x)|\lesssim 2^{-j\lambda}2^{-jd}\left(1+2^{-j}|x|\right)^{-d-1}.$$ 
 Finally, from the definition $K(x)-\tilde{K}(x)=\int_{\R^d}\varphi(|\xi|^{2\lambda})e^{2\pi ix.\xi}d\xi.$ 
 Since $\varphi(|\xi|^{2\lambda})\equiv 1$ near $0$, using integration by parts again we get that $$|K(x)-\tilde{K}(x)|\lesssim_{\lambda} \left(1+|x|\right)^{-d-1}~~\text{for~~ any}~~\lambda>0.$$
 This completes the proof. 
 \qed
 
\subsection*{Proof of theorem \ref{schrodinger}}
Let $\varphi \in C_c^{\infty}(\R)$ be an even function such that $\varphi\equiv 1$ in a neighbourhood of the origin. Write 
$m_{\beta}(u)= m_{\beta,1}(u) +m_{\beta,2}(u),$
where $m_{\beta,1}(u)=m_{\beta}(u)\varphi(u)$. 
Note that we can realise these as functions on $\R^{2n}$ by putting $u=|\xi,\eta|^2$ for $\xi,\eta\in \R^n$. The $2n-$dimensional Fourier transform is given by 
$$\widehat{m_{\beta,1}}(x)=\int_{\R^{2n}}m_{\beta,1}(|(\xi,\eta)|^2)e^{2\pi ix\cdot(\xi+\eta)}d\xi d\eta.$$
We will show the following estimate for $\widehat{m_{\beta,1}}$.  
\begin{eqnarray}\label{app:sch1}|\widehat{m_{\beta,1}}(x)|\lesssim_{\beta} |x|^{-(2n+1)} ~~\text{for}~~|x|~~\text{large}~~ \text{and}~ 0<\beta.
\end{eqnarray}
Consider $$\widehat{m_{\beta,1}}(x)=\int_{\R^{d}}\left(e^{i|\xi|^{2\beta}}-1\right) \varphi\left(|\xi|^2\right)e^{2\pi ix.\xi} d\xi +\int_{\R^{d}}\varphi\left(|\xi|^2\right)e^{2\pi ix.\xi} d\xi .$$
It is easy to verify that  
$$\left|\int_{\R^{d}}\varphi\left(|\xi|^2\right)e^{2\pi ix.\xi} d\xi\right| \lesssim (1+|x|)^{-2n-\beta'}~~\text{for~ any}~~\beta'>0.$$
Next, for the remaining part we perform a partition of unity argument and consider the integrals
$$I_j(x)=\int_{\R^{d}}(e^{i|\xi|^{2\beta}}-1) \phi(2^j\xi)e^{2\pi ix.\xi} d\xi$$
where $\phi$ is a smooth function supported in $1/4\leq |\xi|\leq 4$ and $\sum_{j\geq 0}^{\infty}\phi(2^j\xi)= 1$ on $0<|\xi|\leq 2$.
Note that for $j$ large enough, we know that $\varphi(|\xi|^2)\phi(2^j \xi)=\phi(2^j \xi)$. Therefore, the desired estimate for $\widehat{m_{\beta,1}}(x)$ follows from suitable estimate on $I_j(x)$ for $j$ large. 

Applying a change of variable argument we get that  $$I_j(x)=2^{-jd}\int_{\R^{d}}\left(e^{i(2^{-j}|\xi|)^{2\beta}}-1\right) \phi(\xi)e^{2\pi i2^{-j}x.\xi} d\xi .$$ 
Note that for $j$ large, $2^{-2j\beta}|\xi|^{2\beta}$ is very small on the support of $\phi$ which in turn implies that $\left(e^{i(2^{-j}|\xi|)^{2\beta}}-1\right)$ is very small. Write $e^{i(2^{-j}|\xi|)^{2\beta}}=\sum_{k=0}^{\infty}\frac{1}{k!}(2^{-j}|\xi|)^{2k\beta}$ and use integration by parts argument to get that  $$\left|\int_{\R^{d}} (2^{-j}|\xi|)^{2k\beta}\phi(\xi)e^{2\pi i2^{-j}x.\xi} d\xi\right|\leq c(k,d,\beta)2^{-2jk\beta} 2^{-jd}(1+2^{-j}|x|)^{-d-1},$$ where the constant $c(k,d,\beta)$ is at most a polynomial in $k$ of a fixed degree  for all $k\geq 1$. This estimate gives us 
$$|I_j(x)|\leq 2^{-2j\beta}2^{-jd}\left(1+2^{-j}|x|\right)^{-d-1}.$$
Summing over $j$ implies that 
 
$$\sup_{s>0}|\mathcal T_{m_{\beta,1},s}(f,g)(x)|\lesssim Mf(x)Mg(x).$$

Next, we show that  $\|m_{\beta,2}(\cdot)/(\cdot)^{\gamma}\|_{L^2_{\alpha}}<\infty$ for $\gamma>\beta \alpha$.
Let $\Phi$ is a smooth function supported in $[1,2]$ such that \[\sum_{j\geq 0}\Phi(2^{-j}u)\equiv 1,~~\text{on}~ |u|\geq 1.\] Observe that it suffices to obtain the required estimate on the $L^2_{\alpha}$ norm of  $m^j_{\beta,2}(u)=m_{\beta,2}(u)\phi(2^{-j}u)$ for large $j$. Consider 

\begin{eqnarray*}	\|m^j_{\beta,2}\|^2_{L^2_{\alpha}}
	&=&\int_0^{\infty}u^{2\alpha +1}\left|\left(\frac{d}{du}\right)^{\alpha}\left(\frac{e^{i |\cdot|^{\beta}}}{|\cdot|^{\gamma +1}}\Phi(2^{-j}\cdot)\right)\right|^2 du\\
	&=&2^{-2j\gamma}\int_0^{\infty}u^{2\alpha +1}\left|\left(\frac{d}{du}\right)^{\alpha}\left(e^{i 2^{j\beta}|\cdot|^{\beta}}\widetilde{\Phi}(\cdot)\right)(u)\right|^2 du, \end{eqnarray*} 
where $\widetilde{\Phi}(u)=\frac{\phi(u)}{|u|^{\gamma +1}}$. 

Note that $\widetilde{\Phi}$ is a smooth function supported in $[1,2]$. 
We know that the support of $\left(\frac{d}{du}\right)^{\alpha}\left(e^{i 2^{j\beta}|\cdot|^{\beta}}\widetilde{\Phi}(\cdot)\right)(u)$ is contained in $(-\infty,2]$ when $\alpha>\frac{1}{2}$. Therefore, it is enough to estimate  
\[J_{\alpha}=\int_{\R}\left|\left(\frac{d}{du}\right)^{\alpha}\left(e^{i 2^{j\beta}|\cdot|^{\beta}}\widetilde{\Phi}(\cdot)\right)(u)\right|^2 du.\]

Interpolation between integral values of $\alpha$ gives us that $J_{\alpha}\lesssim 2^{2j\alpha\beta}$. Combining this with the estimates above we get 
\begin{equation}\label{mj}\int_0^{\infty}u^{2\alpha +1}\left|\left(\frac{d}{du}\right)^{\alpha}\left(\frac{e^{i |\cdot|^{\beta}}}{|\cdot|^{\gamma +1}}\Phi(2^{-j}\cdot)\right)\right|^2 du\lesssim 2^{2j(\alpha\beta-\gamma)}.
\end{equation}
This yields that $\|m_{\beta,2}(\cdot)/(\cdot)^{\gamma}\|_{L^2_{\alpha}}<\infty$ when $\gamma>\beta \alpha$.

Next, note that when $\gamma$ is an integer, using binomial expansion we can write  $$\left(-\triangle\right)^{\gamma}(f\otimes g)(x,y)= \sum_{0\leq \mu_1 +\mu_2\leq \gamma}c(\gamma,\mu_1,\mu_2)(-\triangle)^{\mu_1}f(x)(-\triangle)^{\mu_2}g(y).$$ 
Using the estimate \eqref{mj} for $\gamma_0=m$ and $\gamma_1=m+1, m\geq 0,$ we get
$$\sup_{0<s<1}\left|\mathcal T_{m^j_{\beta},s}(f,g)(x)\right| \lesssim 2^{j(\alpha\beta-\gamma_k)}\sum_{0\leq \mu_1 +\mu_2\leq \gamma_k}c(\gamma_k,\mu_1,\mu_2) \left(\mathcal G^{\alpha-1}((-\triangle)^{\mu_1}f,(-\triangle)^{\mu_2}g)(x)\right)$$
where $k=0,1$ and $m\geq 0$.

Invoking Theorem \ref{maintheorem:sqr} we get that 
$$\|\sup_{0<s<1}|\mathcal T_{m^j_{\beta},s}(f,g)|\|_{p}\lesssim 2^{j(\alpha\beta-\gamma_k)}\|(I -\triangle)^{\gamma_k}f\|_{p_1} \|(I -\triangle)^{\gamma_k}g\|_{p_2}$$
for $k=0,1$ and $\alpha>\alpha_*(p_1,p_2)+1$.

An interpolation argumnet (see Theorem 6.4.5 on page 152 and Theorem 4.4.1 on page 96 in \cite{BL}) yields  $$\|\sup_{0<s<1}|\mathcal T_{m^j_{\beta},s}(f,g)|\|_{p}\lesssim 2^{j(\alpha\beta-\gamma)}\|(I -\triangle)^{\gamma}f\|_{p_1} \|(I -\triangle)^{\gamma}g\|_{p_2}$$
for any $m<\gamma<m+1,m\geq 0$ and $j\geq 1$.

When $\gamma>\alpha\beta$ summing over $j$ gives us 
$$\|\sup_{0<s<1}|\mathcal T_{m_{\beta},s}(f,g)|\|_{p}\lesssim \|(I -\triangle)^{\gamma}f\|_{p_1} \|(I -\triangle)^{\gamma}g\|_{p_2},$$ 
where $\alpha>\alpha_*(p_1,p_2)+1$. 
\subsection*{Proof of Theorem~\ref{bilinearradial}}
Recall that we need to prove the required estimates for bilinear maximal functions associated with the following operators 
$$\mathcal T_j^s(f,g)(x)=\int_{\R^{2n}}m_j(s^2|(\xi,\eta)|^2)\hat{f}(\xi)\hat{g}(\eta)e^{2\pi ix\cdot(\xi+\eta)}d\xi d\eta ~~\text{for}~~ j\geq 1~~\text{and}$$
$$\mathcal T_0^s(f,g)(x)=\int_{\R^{2n}}m_0(s^2|(\xi,\eta)|^2)\phi(s^2|(\xi,\eta)|^2)\hat{f}(\xi)\hat{g}(\eta)e^{2\pi ix\cdot(\xi+\eta)}d\xi d\eta.$$ 
First, observe that $m_0(|(\xi,\eta)|^2)\phi(|(\xi,\eta)|^2)$ is a compactly supported smooth function on $\R^{2n}$. Therefore, the corresponding maximal function $\sup_{s>0}|\mathcal T_0^s(f,g)|$ can be dominated by ${M}f(x){M}g(x)$ pointwise a.e. and hence the desired $L^p$ estimate follows. 

Next, we will show that $\sup_{s>0}|\mathcal T_j^s(f,g)|$ extends to a bounded operator on $L^2(\R^n)\times L^2(\R^n)\rightarrow L^1(\R^n)$ and its norm is bounded by $\|m_j\|_{L^2_{\beta}}$ for each $j\geq 1$. 
From the inequality \eqref{app1} we see that $$\sup_{s>0}|\mathcal T_j^s(f,g)(x)|\lesssim \|m_j\|_{L^2_{\beta}} \mathcal{G}^{\beta-1}(f,g)(x)~\text{ for~ a.e.} ~x.$$ 
Using the boundedness of bilinear Stein's square function  (see Theorem \ref{maintheorem:sqr}) we know that $\mathcal{G}^{\beta-1}(f,g)$ is bounded from $L^2(\R^n)\times L^2(\R^n)\rightarrow L^1(\R^n)$ for any $\beta>1$. The given criteria on $m_j$ allows us to sum the R.H.s. in the above inequality to arrive at the conclusion. This completes the proof.
\qed
\subsection*{Proof of Theorem~\ref{classicalmultiplier}}
The proof follows along the same lines as the theorem above. We will borrow the notation from the above theorem. Again it is easy to check that $\mathcal T_0^*$ is bounded from $L^2(\R^n)\times L^2(\R^n)\rightarrow L^1(\R^n)$. For the remaining part using \eqref{app1} and applying it for $\beta=2$ gives us the result. We skip the details here. \qed

\section{Proof of Theorem~\ref{maintheorem:sqr}}\label{sec:pqr}
 We decompose the bilinear Bochner-Riesz square function along the same lines as carried out in ~[Section 3, \cite{JS}]. This involves decomposing the bilinear Bochner-Riesz multiplier  $\left(1-\frac{|\xi|^2+|\eta|^2}{R^2}\right)^{\alpha}_+$ along the $\xi$ ad $\eta$- axes separately. We proceed as follows. 

Let $\psi\in C^\infty_0[\frac{1}{2},2]$ and $\psi_0\in C_0^{\infty}[-\frac{3}{4},\frac{3}{4}]$ be such that  
$$\sum_{j\geq 2}\psi(2^j(1-t))+\psi_0(t)=1~~\text{for~~all}~t\in [0,1).$$ 
This allows us to write 
\begin{eqnarray*}\label{decom}
	\widehat{\K_R^{\alpha}}(\xi,\eta)
	&=& \frac{|\xi|^2+|\eta|^2}{R^2}\left(1-\frac{|\xi|^2+|\eta|^2}{R^2}\right)_+^{\alpha}
	=\sum_{j\geq 2}\m^{\alpha}_{j,R}(\xi,\eta)+\m^{\alpha}_{0,R}(\xi,\eta),
\end{eqnarray*}
where for $j\geq 2$,  $$\m^{\alpha}_{j,R}(\xi,\eta)=\psi\left(2^j\left(1-\frac{|\xi|^2}{R^2}\right)\right)\frac{|\xi|^2+|\eta|^2}{R^2}\left(1-\frac{|\xi|^2}{R^2}\right)_+^{\alpha}\left(1-\frac{|\eta|^2}{R^2}\left(1-\frac{|\xi|^2}{R^2}\right)^{-1}\right)^{\alpha}_+$$ and 
$$\m^{\alpha}_{0,R}(\xi,\eta)=\psi_0\left(\frac{|\xi|^2}{R^2}\right)\frac{|\xi|^2+|\eta|^2}{R^2}\left(1-\frac{|\xi|^2+|\eta|^2}{R^2}\right)_+^{\alpha}.$$
This yields the following decomposition of the operator 
\begin{eqnarray*}\label{decom1}
	\g_R^{\alpha}(f,g)(x)=\sum_{j\geq 2}\g_{j,R}^{\alpha}(f,g)(x)+\g^{\alpha}_{0,R}(f,g)(x),
	\end{eqnarray*}
where 
$$\g^{\alpha}_{j,R}(f,g)(x)=\int_{\R^n}\int_{\R^n}\m^{\alpha}_{j,R}(\xi,\eta)\hat{f}(\xi)\hat{g}(\eta)e^{2\pi ix.(\xi+\eta)}d\xi d\eta, ~j\geq 2$$
and 
$$\g^{\alpha}_{0,R}(f,g)(x)=\int_{\R^n}\int_{\R^n}\m^{\alpha}_{0,R}(\xi,\eta)\hat{f}(\xi)\hat{g}(\eta)e^{2\pi ix.(\xi+\eta)}d\xi d\eta.$$
Consequently, we get that 
\begin{eqnarray}\label{firstdecom}	\G^{\alpha}(f,g)(x)\leq 	\G_0^{\alpha}(f,g)(x)+\sum_{j\geq 2}	\G_j^{\alpha}(f,g)(x)
	\end{eqnarray}
where 
\begin{eqnarray*}\label{gfunc2}
	\G_j^{\alpha}(f,g)(x)
	&=&\left(\int_{0}^{\infty}|\mathcal \g_{j,R}^{\alpha} (f,g)(x)|^2\frac{dR}{R}\right)^{\frac{1}{2}}, j=0, 2, 3, \dots
\end{eqnarray*}
Therefore, our job of proving $L^{p_1}(\R^n)\times L^{p_2}(\R^n)\rightarrow L^{p}(\R^n)$ boundedness of $\G^{\alpha}$ is reduced to obtaining the same for new square functions $\G_j^{\alpha}$ 
with $\| \G_j^{\alpha}\|_{L^{p_1}\times L^{p_2}\rightarrow L^{p}}\leq C_j$ such that $\sum\limits_{j\geq 2} C_j<\infty.$ We will address the problem of boundedness of $\G^{\alpha}_0$ and $\G^{\alpha}_j, j\geq 2$ separately.

\subsection*{Boundedness of $\G^{\alpha}_j, j\geq 2$:} 

We invoke the decomposition of the bilinear Bochner-Riesz multiplier $\psi\left(2^j\left(1-\frac{|\xi|^2}{R^2}\right)\right)\left(1-\frac{|\xi|^2+|\eta|^2}{R^2}\right)_+^{\alpha}, \alpha>0$ from~[Section 3, \cite{JS}] to write that 
\begin{eqnarray}\label{decom3} \g^{\alpha}_{j,R}(f,g)(x)
	&=&\label{decomope}c_{\alpha}R^{-2\alpha}\int_0^{R_j}\left(B_{j,\beta}^{R,t}f(x)B_t^{\delta}g(x) +A_{j,\beta}^{R,t}f(x)A_t^{\delta}g(x) \right)t^{2\delta+1}dt
\end{eqnarray} 
where $\alpha=\beta+\delta, \beta>\frac{1}{2}, \delta>-\frac{1}{2}, R_j=R\sqrt{2^{-j+1}}, \varphi_R(\xi)=\left(1-\frac{|\xi|^2}{R^2}\right)_+$ and
\begin{eqnarray}\label{constant:stein}
	c_{\alpha}=2\frac{\Gamma(\delta+\beta+1)}{\Gamma(\delta+1)\Gamma(\beta)},
\end{eqnarray}
Further,  the operators in the right hand side of \eqref{decom3} are defined by
\begin{eqnarray*}\label{B_j}B_{j,\beta}^{R,t}f(x)=\int_{\R^n}\hat{f}(\xi)\psi\left(2^j\left(1-\frac{|\xi|^2}{R^2}\right)\right)\frac{|\xi|^2}{R^2}\left(R^2\varphi_R(\xi)-t^2\right)_+^{\beta-1}e^{2\pi i x.\xi}d\xi
\end{eqnarray*} 

\begin{eqnarray*}\label{A_j}A_{j,\beta}^{R,t}f(x)=\int_{\R^n}\hat{f}(\xi)\psi\left(2^j\left(1-\frac{|\xi|^2}{R^2}\right)\right)\left(R^2\varphi_R(\xi)-t^2\right)_+^{\beta-1}e^{2\pi i x.\xi}d\xi
\end{eqnarray*}
and \begin{eqnarray*}\label{AR}A_t^{\delta}g(x)=\int_{\R^n}\hat{g}(\eta)\frac{|\eta|^2}{R^2}\left(1-\frac{|\eta|^2}{t^2}\right)^{\delta}_+e^{2\pi ix.\eta} d\eta.
\end{eqnarray*}
Note that $B_t^{\delta}$ in~\eqref{decom3} denotes the classical Bochner-Riesz means of index $\delta$.
Consider the equation ~\eqref{decom3} and write it as 
\begin{eqnarray}\label{decom4} \g^{\alpha}_{j,R}(f,g)(x)=I_{j,R}+II_{j,R}, 
\end{eqnarray}
where 
\begin{eqnarray*}I_{j,R}=c_{\alpha}R^{-2\alpha}\int_0^{R_j}B_{j,\beta}^{R,t}f(x)B_t^{\delta}g(x) t^{2\delta+1}dt
\end{eqnarray*}
and 
$$II_{j,R}=c_{\alpha}R^{-2\alpha}\int_0^{R_j} A_{j,\beta}^{R,t}f(x)A_t^{\delta}g(x)t^{2\delta+1}dt.$$

First, we estimate the square function corresponding to the term $II_{j,R}$. Apply Cauchy-Schwarz inequality  to get that 
\begin{eqnarray*} 
	|II_{j,R}|
	&\lesssim&\nonumber R^{-2\alpha}\left(\int_{0}^{R_j}|A_{j,\beta}^{R,t}f(\cdot)t^{2\delta+1}|^2dt\right)^{1/2}  \left(\int_{0}^{R_j}|A_{Rt}^{\delta}g(x)|^2dt\right)^{1/2}
\end{eqnarray*}

Making a change of variable $t\rightarrow Rt$ in the integral $\left(\int_0^{R_j}|A_{j,\beta}^{R,t}f(x)t^{2\delta+1}|^2 dt\right)^{\frac{1}{2}}$ we get that $$\int_0^{R_j}|A_{j,\beta}^{R,t}f(x)t^{2\delta+1}|^2 dt=R^{4\alpha-1}\int_0^{\sqrt{2^{-j+1}}}|\tilde{S}_{j,\beta}^{R,t}f(x)t^{2\delta+1}|^2 dt,$$
where $$\tilde{S}_{j,\beta}^{R,t}f(x)=\int_{\R^n}\psi\left(2^j\left(1-\frac{|\xi|^2}{R^2}\right)\right)\left(1-\frac{|\xi|^2}{R^2}-t^2\right)_+^{\beta-1}\hat{f}(\xi)e^{2\pi ix.\xi}d\xi.$$ 
We also make the change $t\rightarrow Rt$ in the other integral involving the term $A_t^{\delta}g$ and apply H\"{o}lder's inequality with $\frac{p_1}{p}$ and $\frac{p_2}{p}$ to get that 
\begin{eqnarray*} 
	&&	\left\|\left(\int_0^\infty|II_{j,R}|^2\frac{dR}{R}\right)^{\frac{1}{2}}\right\|_{p}\\
	&\lesssim&\nonumber  \left\|\sup_{R>0}\left(\int_{0}^{\sqrt{2^{-j+1}}}|\tilde{S}_{j,\beta}^{R,t}f(\cdot)t^{2\delta+1}|^2dt\right)^{\frac{1}{2}}\right\|_{p_1}  \left\|\left[\int_0^\infty\left(\int_{0}^{\sqrt{2^{-j+1}}}|A_{Rt}^{\delta}g(x)|^2dt\right)\frac{dR}{R}\right]^{\frac{1}{2}}\right\|_{p_2}
\end{eqnarray*}
We know that the term involving maximal function in the estimate above satisfies the following $L^p$ estimates, see \cite[Theorem 5.1]{JS}.
\begin{eqnarray*}\label{part2}\left\|\sup_{R>0}\left(\int_{0}^{\sqrt{2^{-j+1}}}|\tilde{S}_{j,\beta}^{R,t}f(\cdot)t^{2\delta+1}|^2dt\right)^{\frac{1}{2}}\right\|_{p_1}\lesssim2^{j(\alpha(p_1)+\frac{1}{4}-\alpha+\epsilon)}\|f\|_{p_1},
\end{eqnarray*}
for $n\geq2$ and $\beta>\alpha(p_1)+\frac{1}{2}$. 

The remaining term is dealt with by using $L^p$ estimates for the Bochner-Riesz square function in the following way. 
\begin{eqnarray} 
	\left\|\left[\int_0^\infty\left(\int_{0}^{\sqrt{2^{-j+1}}}|A_{Rt}^{\delta}g(x)|^2dt\right)\frac{dR}{R}\right]^{\frac{1}{2}}\right\|_{p_2} &=&\nonumber  \left\|\left[\int_{0}^{\sqrt{2^{-j+1}}}\left(\int_0^\infty|A_{Rt}^{\delta}g(x)|^2\frac{dR}{R}\right)dt\right]^{\frac{1}{2}}\right\|_{p_2}\\
	&\lesssim &\nonumber 2^{\frac{5}{4}(-j+1)}\left\|\left[\int_0^\infty|A_{R}^{\delta}g(x)|^2\frac{dR}{R}\right]^{\frac{1}{2}}\right\|_{p_2}\\
	&\lesssim &\nonumber 2^{\frac{5}{4}(-j+1)}\|g\|_{p_2},
\end{eqnarray}
where $p_2\geq \min\{\mathfrak p_n,\frac{2(n+2)}{n}\}$ and $\delta>\alpha(p_2)-\frac{1}{2}$. 

Therefore, for $\beta>\alpha(p_1)+\frac{1}{2}$, $\delta>\alpha(p_2)-\frac{1}{2}$ and $p_1, p_2\geq \min\{\mathfrak p_n,\frac{2(n+2)}{n}\}$ , we get
\begin{eqnarray}\label{mainestiforII}\left\|\left(\int_0^\infty|II_{j,R}|^2\frac{dR}{R}\right)^{\frac{1}{2}}\right\|_{p}\lesssim2^{j(\alpha(p_1)+\frac{1}{4}-\alpha+\epsilon)+\frac{5}{4}(-j+1)}\|f\|_{p_1}\|g\|_{p_2}.
\end{eqnarray}

Next, we need to deal with the bilinear square function associated with the term $I_{j,R}$. We begin in a similar fashion as in the previous case of the term $II_{j,R}$. Making use of change of variable $t\rightarrow Rt$ and H\"{o}lder inequality we arrive at the following. 
\begin{eqnarray*} 
	&&\left\|\left(\int_0^\infty |I_{j,R}|^2\frac{dR}{R}\right)^{\frac{1}{2}}\right\|_{p} \\
	&\lesssim &\nonumber  \left\|\left[\int_0^\infty\left(\int_0^{R_j}|B_{j,\beta}^{R,t}f(x)t^{2\delta+1}|^2 dt\right)\left(\int_0^{R_j}|B_t^{\delta}g(x)|^2dt\right)\frac{dR}{R^{4\alpha+1}}\right]^{\frac{1}{2}}\right\|_{p} \\
	&\lesssim&\nonumber 2^{-\frac{j}{4}} \left\|\left[\int_0^\infty\left(\int_{0}^{\sqrt{2^{-j+1}}}|{S}_{j,\beta}^{R,t}f(\cdot)t^{2\delta+1}|^2dt\right)\frac{dR}{R}\right]^{\frac{1}{2}}\right\|_{p_1} \left\|\sup_{R>0}\left(\frac{1}{R_j}\int_0^{R_j}|B_t^{\delta}g(x)|^2dt\right)^{\frac{1}{2}}\right\|_{p_2},
\end{eqnarray*}
where $$S_{j,\beta}^{R,t} f(x)=\int_{\R^n}\psi\left(2^j\left(1-\frac{|\xi|^2}{R^2}\right)\right)\frac{|\xi|^2}{R^2}\left(1-\frac{|\xi|^2}{R^2}-t^2\right)_+^{\beta-1}\hat{f}(\xi)e^{2\pi ix.\xi}d\xi.$$
Invoking the $L^p$ estimates for the maximal operator $f\rightarrow \sup_{R>0}\left(\frac{1}{R}\int_0^{R}|B_t^{\delta}f(\cdot)|^2dt\right)^{\frac{1}{2}}$  from \cite[Lemma 4.4]{JS}, for $\delta>\alpha(p_2)-\frac{1}{2}$ we get that 
$$\left\|\sup_{R>0}\left(\frac{1}{R_j}\int_0^{R_j}|B_t^{\delta}g(x)|^2dt\right)^{\frac{1}{2}}\right\|_{p_2}\lesssim \|g\|_{p_2}.$$
For the other term using Minkowski integral inequality we can write  
\begin{equation}\label{gjj21}
	\left\|\left[\int_0^\infty\left(\int_{0}^{\tau_j}|{S}_{j,\beta}^{R,t}f(\cdot)t^{2\delta+1}|^2dt\right)\frac{dR}{R}\right]^{\frac{1}{2}}\right\|_{p_1}\leq \left(\int_{0}^{\tau_j}\left[\int_{\R^n}\left(\int_0^\infty|{S}_{j,\beta}^{R,t}f(\cdot)|^2\frac{dR}{R}\right)^{\frac{p_1}{2}}dx\right]^{\frac{2}{p_1}}t^{4\delta+2}dt
	\right)^{\frac{1}{2}}.
\end{equation}
Here $\tau_j=\sqrt{2^{-j+1}}$. Therefore, we need to prove the desired estimates for the square function corresponding to operators $S_{j,\beta}^{R,t}$.

Let $\psi\in C^{N+1}([\frac{1}{2},2]), n\geq 1$ and $0<\nu <\frac{1}{16}$. Consider the following operators $$B^{\psi}_{\nu,t}f(x)=\int_{\R^n}\psi\left(\nu^{-1}\left(1-\frac{|\xi|^2}{t^2}\right)\right)\hat{f}(\xi)e^{2\pi ix.\xi}d\xi.$$
The following result is proved in \cite{JS}, also see Jeong, Lee and Vargas~\cite[Lemma $2.6$]{JLV} for a similar result. 
\begin{lemma}\cite{JS}\label{localisedversion} Let $n\geq 2, 0<\nu<\frac{1}{16}$ and $\epsilon>0$. Then for $p\geq \mathfrak p_n$ and $p=2,$ there exists $N\geq 1$ such that for all $\psi\in C^{N+1}([\frac{1}{2},2])$ the following holds 
	\begin{equation*}\label{localised2}\left\|\left(\int_{0}^{\infty}\left|B^{\psi}_{\nu,t}f(\cdot)\right|^2 \frac{dt}{t}\right)^{\frac{1}{2}}\right\|_{L^p(\R^n)}\lesssim_{\epsilon,N} \nu^{(\frac{1}{2}-\alpha(p))}\nu^{-\epsilon}\|f\|_{L^p(\R^n)}
	\end{equation*}
	where the implicit constant depends on $\epsilon$ and $N$.
\end{lemma}
We will make use of Lemma~\ref{localisedversion} to prove the following estimate for the square function associated with the operators ${S}_{j,\beta}^{R,t}$. 
\begin{theorem}\label{maxsquare}
	Let $n\geq 2.$ When $\beta> \alpha(p)+\frac{1}{2},$ then the following estimate holds 
	\begin{eqnarray*}\label{reducedoperator} \left\|\left(\int_0^\infty|{S}_{j,\beta}^{R,t}f(\cdot)|^2\frac{dR}{R}\right)^{\frac{1}{2}}\right\|_{L^p(\R^n)}\lesssim 2^{j(\alpha(p)-\alpha-\frac{1}{4}+\epsilon)}\|f\|_{L^p(\R^n)}
	\end{eqnarray*}
	where the implicit constant is independent of $t\in [0,\sqrt{2^{-j+1}}]$ when $p\geq \mathfrak p_n$ and $p=2$. \end{theorem}
\begin{proof}
	Write $\beta=\gamma+1$ in ${S}_{j,\beta}^{R,t}f$ and note that when $t^2\in[0,2^{-j-1})$ we can rewrite the function 
	$$\psi\left(2^j\left(1-\frac{|\xi|^2}{R^2}\right)\right)\left(1-\frac{|\xi|^2}{R^2}-t^2\right)_+^{\gamma}\frac{|\xi|^2}{R^2}=\psi\left(2^j\left(1-\frac{|\xi|^2}{R^2}\right)\right)\left(1-\frac{|\xi|^2}{R^2}-t^2\right)^{\gamma}\frac{|\xi|^2}{R^2}.$$
	Let $0<\epsilon_0<<1$ be a fixed number and split the interval $[0,2^{-j+1}]$ into  subintervals $[0,2^{-j-1-\epsilon_0}]$ and $[2^{-j-1-\epsilon_0},2^{-j+1}]$. We will deal with the operators corresponding to each subinterval separately. 
	\subsection*{Case I: When $t\in[0,\sqrt{2^{-j-1-\epsilon_0}}]$}
	First, consider the range $-1<\gamma<0.$ Write $\gamma=-\rho$, where $\rho>0$. Note that using Taylor's expansion we can write 
	\begin{eqnarray*}\label{taylorexpansion}
		\left(1-\frac{|\xi|^2}{R^2}\right)^{-\rho}\left(1-\frac{t^2}{1-\frac{|\xi|^2}{R^2}} \right)^{-\rho}\frac{|\xi|^2}{R^2}&=&2^{j\rho}\left(2^{j}\left(1-\frac{|\xi|^2}{R^2}\right)\right)^{-\rho}\sum_{k\geq 0}\frac{\Gamma(\rho+k)}{\Gamma(\rho)k!}\left(\frac{2^jt^2}{2^j(1-\frac{|\xi|^2}{R^2})}\right)^k\\
		&&-\nonumber 2^{j(\rho-1)}\left(2^{j}\left(1-\frac{|\xi|^2}{R^2}\right)\right)^{-\rho+1}\sum_{k\geq 0}\frac{\Gamma(\rho+k)}{\Gamma(\rho)k!}\left(\frac{2^jt^2}{2^j(1-\frac{|\xi|^2}{R^2})}\right)^k.
	\end{eqnarray*}
	Since $\frac{t^2}{1-\frac{|\xi|^2}{R^2}}\leq 2^{-j-1-\epsilon_0}2^{j+1}<1,$ the series in the estimate above converges. Thus,  we get that
	\begin{eqnarray*}\label{s_j}{S}_{j,\gamma+1}^{R,t}f&=&2^{j\rho}\sum_{k\geq0}\frac{\Gamma(\rho+k)}{\Gamma(\rho)k!}\left(2^jt^2\right)^kB^{\psi^k}_{2^{-j},R}f-2^{j(\rho-1)}\sum_{k\geq0}\frac{\Gamma(\rho+k)}{\Gamma(\rho)k!}\left(2^jt^2\right)^kB^{\psi'^k}_{2^{-j},R}f,
	\end{eqnarray*}
	where $\psi^k(x):= x^{-k-\rho}\psi(x)$ and $\psi'^k(x):= x^{-k-\rho+1}\psi(x)$. 
	
	Observe that $\psi^k\in C^{\infty}_0([\frac{1}{2},2])$ and it satisfies the estimate
	$$\sup_{x\in[\frac{1}{2},2],0\leq l\leq N}\left|\frac{d^l\psi^k}{dx^l}\right|\leq C(\rho)2^{N+k} k^{N+1}$$
	for $N\geq 20n$. Similar estimates hold for $\psi'^k$.
	
	Lemma \ref{localisedversion} gives us 
	$$\left\|\left(\int_0^\infty|B^{\psi^k}_{2^{-j},R}f|^2\frac{dR}{R}\right)^{\frac{1}{2}}\right\|_{p}\lesssim_{N,p}2^{N+k} k^{N}2^{j\alpha(p)-j/2}\|f\|_{p}.$$
	Consequently, we get that 
	\begin{eqnarray*}
		\left\|\left(\int_0^\infty|{S}_{j,\gamma+1}^{R,t}f(\cdot)|^2\frac{dR}{R}\right)^{\frac{1}{2}}\right\|_{p}&\lesssim& 2^{j\rho}\sum_{k\geq0}\frac{\Gamma(\rho+k)}{\Gamma(\rho)k!}\left(2^jt^2\right)^k\left\|\left(\int_0^\infty|B^{\psi^k}_{2^{-j},R}f|^2\frac{dR}{R}\right)^{\frac{1}{2}}\right\|_{p}\\
		&\lesssim& 2^{j(\rho+\alpha(p))-j/2}\|f\|_{p}\sum_{k\geq0}\frac{\Gamma(\rho+k)}{\Gamma(\rho)k!}\left(2^jt^2\right)^k2^{N+k} k^{N}
	\end{eqnarray*}
	Since $\frac{\Gamma(\rho+k)}{\Gamma(\rho)k!}\approx k^{\rho}$
	and $t^2\leq 2^{-j-1-\epsilon_0}$, the series in the expression above converges.
	
	This completes the proof of Theorem~\ref{maxsquare} for the range $-1<\gamma<0.$ When $\gamma\geq 0$, the proof follows in a similar manner as in Theorem $5.1$ in \cite{JS}. The details are left to the reader. 
	\subsection*{Case II: When $t\in [\sqrt{2^{-j-1-\epsilon_0}},\sqrt{2^{-j+1}}]$}
	
	Note that in this case $t\approx 2^{-j/2}.$ We rewrite
	$$\psi\left(2^j\left(1-\frac{|\xi|^2}{R^2}\right)\right)\left(1-\frac{|\xi|^2}{R^2}-t^2\right)_+^{\gamma}\frac{|\xi|^2}{R^2} = (1-t^2)^{\gamma+1}\psi\left(2^j\left(1-\frac{|\xi|^2}{R^2}\right)\right)\left(1-\frac{|\xi|^2}{R^2(1-t^2)}\right)_+^{\gamma}
	\frac{|\xi|^2}{R^2(1-t^2)}.$$
	
	Let $1-t^2=s^2$ and consider the operator ${S}_{j,\gamma+1}^{R,s}$.
	
	$${S'}_{j,\gamma+1}^{R,s}f(x)=\int_{\R^n}\psi\left(2^j\left(1-\frac{|\xi|^2}{R^2}\right)\right)\left(1-\frac{|\xi|^2}{s^2R^2}\right)_+^{\gamma}\frac{|\xi|^2}{s^2R^2}\hat{f}(\xi)e^{2\pi ix.\xi}d\xi,$$ 
	where $s\in \left[\sqrt{1-2^{-j+1}},\sqrt{1-2^{-j-1-\epsilon_0}}\right]$.
	
	Note that we need to establish the $L^p$ estimates for the square function $\left(\int_0^\infty|{S'}_{j,\gamma+1}^{R,s}f(\cdot)|^2\frac{dR}{R}\right)^{\frac{1}{2}}$. Let $\Lambda>100$ be a large number. By considering the cases $j\geq \Lambda$ and $2\leq j<\Lambda$ separately we have the following results. 
	\begin{proposition}\label{maxsquare2} For $j\geq \Lambda, 0<\epsilon<1$ and $\gamma>\alpha(p)-\frac{1}{2}$ the estimate 
		$$\sup_{s\in[s_1,s_2]}\left\|\left(\int_0^\infty|{S'}_{j,\gamma+1}^{R,s}f(\cdot)|^2\frac{dR}{R}\right)^{\frac{1}{2}}\right\|_{L^p(\R^n)}\lesssim_{\epsilon} 2^{-j\gamma }2^{j(\alpha(p)-\frac{1}{2})}2^{\epsilon j}\|f\|_{L^p(\R^n)},$$ 
		holds when $p\geq \mathfrak p_n$ and $p=2$ where $s_1=\sqrt{1-2^{-j+1}}$ and $s_2=\sqrt{1-2^{-j-1-\epsilon_0}}$. 
	\end{proposition}
	\begin{proposition}\label{last} When $\gamma>\alpha(p)-\frac{1}{2}$ and $0<\epsilon <1$ we have 
		$$\sup_{s\in [s_1,s_2]}\left\|\left(\int_0^\infty|{S'}_{j,\gamma+1}^{R,s}f(\cdot)|^2\frac{dR}{R}\right)^{\frac{1}{2}}\right\|_{L^p(\R^n)}\leq C(\Lambda,\gamma,\epsilon)\|f\|_{L^p(\R^n)}$$
		for all $2\leq j\leq \Lambda$, $p\geq \mathfrak p_n$ and $p=2$.
	\end{proposition}
	
	\subsection*{Proof of Proposition~\ref{last}}
	Note that \begin{eqnarray*}
		{S'}_{j,\gamma+1}^{R,s}f(x)&=&\int_{\R^n}\psi\left(2^j\left(1-\frac{|\xi|^2}{R^2}\right)\right)\left(1-\frac{|\xi|^2}{R^2s^2}\right)_+^{\gamma}\frac{|\xi|^2}{s^2R^2}\hat{f}(\xi)e^{2\pi ix.\xi}d\xi\\
		&=&B^{\psi}_{2^{-j},R}B_{Rs}^{\gamma}f(x),
	\end{eqnarray*}
	where  $$B^{\psi}_{2^{-j},R}(f)(x)=\int_{\R^n}\psi\left(2^j\left(1-\frac{|\xi|^2}{R^2}\right)\right)\hat{f}(\xi)e^{2\pi ix.\xi} d\xi$$
	and 
	$$B^{\gamma}_{Rs}(f)(x)=\int_{\R^n}\left(1-\frac{|\xi|^2}{R^2s^2}\right)_+^{\gamma}\frac{|\xi|^2}{R^2s^2}\hat{f}(\xi)e^{2\pi ix.\xi} d\xi.$$
	
	Since $\psi\left(2^j\left(1-\frac{|\xi|^2}{R^2}\right)\right)$ is a compactly supported smooth function, we can use integration by parts to show that 
	$$\sup_{R>0}|B^{\psi}_{2^{-j},R}f(x)|\leq 2^{jn}Mf(x).$$ 
	Here $M$ denotes the classical Hardy-Littlewood maximal function
	
	This implies that  
	$$\left\|\left(\int_0^\infty|B^{\psi}_{2^{-j},R/s}B_{R}^{\gamma}f(\cdot)|^2\frac{dR}{R}\right)^{\frac{1}{2}}\right\|_{p}\lesssim 2^{jn}\left\|\left(\int_0^\infty|MB^{\gamma}_{Rs}f(\cdot)|^2\frac{dR}{R}\right)^{\frac{1}{2}}\right\|_{p}.$$
	The $L^p$ boundedness of the vector-valued extension of the Hardy-Littlewood Maximal function $M$ (see~\cite{DK, FS}) in the inequality above gives us 
	$$\left\|\left(\int_0^\infty|{S'}_{j,\gamma+1}^{R,s}f(\cdot)|^2\frac{dR}{R}\right)^{\frac{1}{2}}\right\|_{p}\lesssim 2^{jn}\left\|\left(\int_0^\infty|B^{\gamma}_{Rs}f(\cdot)|^2\frac{dR}{R}\right)^{\frac{1}{2}}\right\|_{p}.$$
	
	Finally, using $L^p$ boundedness of Bochner-Riesz square function when   $\gamma>\alpha(p)-\frac{1}{2}$ for $p\geq \mathfrak p_n$ and $p=2$ we get the desired result.  
\end{proof}
\subsection*{Proof of Proposition \ref{maxsquare2}}
In order to prove our estimates we exploit the method from~\cite[Proposition $5.2$]{JS}. For the convenience of the reader and an easy reference we use the same notation as in \cite{JS}. 

Let $R=1$ and write 
$$\left(1-\frac{|\xi|^2}{s^2}\right)_+^{\gamma}\frac{|\xi|^2}{s^2}=\sum_{k\geq 2}2^{-k\gamma}\tilde{\psi}\left(2^k\left(1-\frac{|\xi|^2}{s^2}\right)\right)\frac{|\xi|^2}{s^2} +\tilde{\psi}_0\left(\frac{|\xi|^2}{s^2}\right)\frac{|\xi|^2}{s^2},$$
where $\tilde{\psi}\in C^{\infty}_0([\frac{1}{2},2])$ and $\tilde{\psi}_0\in C^{\infty}_0([0,3/4])$. 

Observe that the corresponding expression for $R\neq 1$ can be written by replacing $\xi$ with $\xi/R$ in the equation above.

First, observe that the product  $\psi\left(2^j\left(1-|\xi|^2\right)\right)\tilde{\psi}\left(2^k\left(1-\frac{|\xi|^2}{s^2}\right)\right)$ vanishes if $k<j-2$ and $s\in \left[\sqrt{1-2^{-j+1}},\sqrt{1-2^{-j-1-\epsilon_0}}\right]$. Therefore, we can write
\begin{eqnarray*}
	\psi\left(2^j\left(1-|\xi|^2\right)\right)\left(1-\frac{|\xi|^2}{s^2}\right)_+^{\gamma}\frac{|\xi|^2}{s^2}=\sum_{k\geq j-2}2^{-k\gamma}\tilde{\psi}\left(2^k\left(1-\frac{|\xi|^2}{s^2}\right)\right)\frac{|\xi|^2}{s^2}\psi\left(2^j\left(1-|\xi|^2\right)\right)\\
\end{eqnarray*} 
Further, since $\frac{|\xi|^2}{s^2}\approx 1$ on the support of $\tilde{\psi}\left(2^k\left(1-\frac{|\xi|^2}{s^2}\right)\right)$ for all $k\geq \Lambda-2$, the term $\frac{|\xi|^2}{s^2}$ can be ignored. Therefore, it is enough to work with the multiplier
$$\psi\left(2^j\left(1-|\xi|^2\right)\right)\left(1-\frac{|\xi|^2}{s^2}\right)_+^{\gamma}=\sum_{k\geq j-2}2^{-k\gamma}\tilde{\psi}\left(2^k\left(1-\frac{|\xi|^2}{s^2}\right)\right)\psi\left(2^j\left(1-|\xi|^2\right)\right).$$
We need to further decompose each piece   $\tilde{\psi}\left(2^k\left(1-\frac{|\xi|^2}{s^2}\right)\right)\psi\left(2^j\left(1-|\xi|^2\right)\right), k\geq j-2$. We refer to~\cite[Equation $(29)$]{JS} for this decomposition and in order to avoid repetition and new notations we directly import it from there. It goes as follows. 

For $0<\epsilon<1$ let $\nu=\frac{2}{3}2^{-k\epsilon}, \xi_{k,m}^s=\left(\frac{|\xi|^2}{s^2}-1+2^{-k+1}-2^{-(1+\epsilon)k} m\right).$ Write $2^k(2^{-k+1}+2^{-k(1+\epsilon)}m)=d^{\epsilon}_{k,m}$. Then we have 
\begin{eqnarray}\label{microdecomposition}
	{S'}_{j,\gamma+1}^{R,s}f(x)&=&\sum_{k\geq j-2}2^{-k\gamma}\sum_{0\leq m\leq [\nu^{-1}]+1}\sum_{q=0}^{N-1}\frac{(-1)^q}{q!}(2\pi i2^{-\epsilon k})^q\frac{d^q\tilde{\psi}}{dx^q}(d^{\epsilon}_{k,m})U^{\varphi_q,j}_{Rs,R}(k,m)f(x)\\ 
	&&\nonumber +\sum_{k\geq j-2}2^{-k\gamma}\sum_{m\leq [\nu^{-1}]+1}\int_{\R}\widehat{\tilde{\psi}}(\mu)e^{2\pi i d^{\epsilon}_{k,m}\mu}P_{Rs,R}^{j,k}(\mu,m)f(x)d\mu,
\end{eqnarray}
where \begin{equation*}\label{opertorU}U^{\varphi_q,j}_{Rs,R}(k,m)f(x)=\int_{\R^n}\varphi_q\left(2^{(1+\epsilon)k}\xi_{k,m}^{Rs}\right)\psi\left(2^j\left(1-\frac{|\xi|^2}{R^2}\right)\right)\hat{f}(\xi)e^{2\pi ix.\xi} d\xi,
\end{equation*} and
\begin{equation*}\label{operatorP}P_{Rs,R}^{j,k}(\mu,m)f(x)=\int_{\R^n}\varphi\left(2^{(1+\epsilon)k}\xi_{k,m}^{Rs}\right)\psi\left(2^j\left(1-\frac{|\xi|^2}{R^2}\right)\right)r_N\left(2^{k}\xi_{k,m}^{Rs}\mu\right)\hat{f}(\xi)e^{2\pi ix.\xi}d\xi.
\end{equation*}

In the expression \eqref{microdecomposition} we have used the notation  $\varphi_q(x)=x^q\varphi(x)$ where $\varphi\in C_c^{\infty}([-1,1])$ is such that $\sum_{m\in\Z}\varphi(x-m)=1$ on $\R.$ Let $L= \sup_{0\leq l\leq N, x\in [-1,1]}\left|\frac{d^l\varphi}{dx^l}\right|.$ The remainder term $r_N$ satisfies the estimate 
\begin{equation*}\label{remainderagain}\left|\frac{d^qr_N}{ds^q}(s)\right|\leq s^{N-q}.
\end{equation*}

Observe that in order to prove Proposition~~\ref{maxsquare2}
it suffices to obtain required estimates for square functions associated with $U^{\varphi_q,j}_{Rs,R}(k,m)$ and $P_{Rs,R}^{j,k}(\mu,m)$. Since $\varphi_q$ behaves the same way for all $0\leq q\leq N-1$, it is enough to  describe proof for the operator $U^{\varphi_q,j}_{Rs,R}(k,m)$ with $q=0$. Also, note that $\varphi_q$ is a smooth function supported in $[-1,1]$ and it satisfies the estimate  
$$\sup_{0\leq l\leq N, x\in [-1,1]}|\frac{d^l\varphi_q}{dx^l}|\leq L q!.$$  

Let $R=1$ (for $R\neq 1$ replace $\xi$ by $\xi/R$) and decompose  $\psi(2^j(1-|\xi|^2))$ into smooth functions having their supports in the annulus $\{\xi: |\xi|^2\in [a+ s^22^{-(1+\epsilon)k}(m-1),a+ s^22^{-(1+\epsilon)k}(m+1)]\}, 0\leq m\leq [\nu^{-1}]+1$, where $a=s^2(1-2^{-k+1}).$  Consider the following operators  $$V^{\varphi_l}_{Rs}(k,m)f(x)=\int_{\R^n}\varphi_l\left(2^{(1+\epsilon)k}\xi_{k,m}^{Rs}\right)\hat{f}(\xi)e^{2\pi ix.\xi} d\xi,$$
and 
$$Q_{N,R,s}^{k}(\mu)f(x)=\int_{\R^n}\varphi\left(2^{(1+\epsilon)k}\xi_{k,m}^{Rs}\right)r_N\left(2^js^2\xi_{k,m}^{Rs}\mu\right)\hat{f}(\xi)e^{2\pi ix.\xi}d\xi.$$
Equation (30) in \cite{JS} allows us to write  
\begin{eqnarray}\label{decompositionofU}
	U^{\varphi_0,j}_{Rs,R}(k,m)f(x)&=&\sum_{l\geq 0}^{N-1}\frac{1}{l!}s^22^{-d(1+\epsilon)l}2^{-j\epsilon l}\psi^l(2^j(1-a-s^22^{-(1+\epsilon)k}m))V^{\varphi_l}_{Rs}(k,m)f(x) \\
	&&\nonumber + X_{N,R,s}^kf(x),
\end{eqnarray}
where 
\begin{eqnarray}\label{Xope}
	X_{N,R,s}^kf(x)=\int_{\R}\hat{\psi}(\mu)e^{2\pi i2^j(1-a-s^22^{-(1+\epsilon)k}m)\mu}Q_{N,R,s}^{k}(\mu)f(x) d\mu
\end{eqnarray} and $\varphi_l(x)=x^l \varphi(x).$

In~\cite[Lemma $5.5$]{JS} authors proved $L^p$ estimates for a maximal function variant corresponding to the operator $V^{\varphi_l}_{Rs}(k,m)$. With suitable modifications in their proof we get the following estimate for the squre function
\begin{eqnarray}\label{reduction1}\left\|\left(\int_{0}^{\infty}|V^{\varphi_l}_{Rs}(k,m)f(\cdot)|^2\frac{dR}{R}\right)^{\frac{1}{2}}\right
	\|_{p}\leq C(p)l!2^{(1+\epsilon)(\alpha(p)-\frac{1}{2})k}\|f\|_{p}
\end{eqnarray}
where $p$ and $\alpha(p)$ satisfy the assumption of Proposition~\ref{maxsquare2}, i.e., $p\geq \mathfrak p_n$ and $\gamma>\alpha(p)-\frac{1}{2}$. Moreover, the implicit constant in the inequality above is uniform with respect to $s\in [s_1,s_2]$. We skip the details. 

Next, for the square function corresponding to $X^k_{N,R,s}$ we have the the following estimate. 
\begin{lemma}\label{reduction2} For $1<p< \infty$, we have 
	$$\sup_{s\in [s_1,s_2]}\left\|\left(\int_{0}^{\infty}|X^k_{N,R,s}f|^2\frac{dR}{R}\right)^{\frac{1}{2}}\right\|_{p}\lesssim 2^{(1+\epsilon)(\alpha(p)-\frac{1}{2})k}\|f\|_{p}.$$
\end{lemma}
\begin{proof} 	
	Recall the definition of $X^k_{N,R,s}$ from equation~\eqref{Xope} and note that it is enough to show that 
	\begin{equation}\label{Qoper}\sup_{s\in [s_1,s_2]}\left\|\left(\int_{0}^{\infty}|Q_{N,R,s}^{k}f|^2\frac{dR}{R}\right)^{\frac{1}{2}}\right\|_{p}\lesssim (1+|\mu|)^{N}2^{-j\epsilon N}2^{-d(1+\epsilon)N}2^{(1+\epsilon)k(n+1)}\|f\|_{p}.\end{equation}
	where $k=j+d$.
	
	Define 
	\begin{equation*}\label{multiplierform}
		M_{s,\mu}(\xi)=2^{j\epsilon N}2^{d(1+\epsilon)N}\varphi\left(2^{(1+\epsilon)k}\xi_{k,m}^s\right)r_N\left(2^js^2\xi_{k,m}^s\mu\right)
	\end{equation*}
	and write $M_{R,s,\mu}(\xi)=M_{s,\mu}(\xi/R)$. 
	Note that 
	$$Q_{N,R,s}^{k}(\mu)f=2^{-j\epsilon N}2^{-d(1+\epsilon)N} f* (\mathcal{F}^{-1}M_{R,s,\mu}).$$ 
	Using the definition of $\xi_{k,m}^s$ and $\varphi$, one can verify that support of $M_{R,s,\mu}(\xi)$ is contained in the set $\{\xi:\frac{|\xi|}{R}\approx 1\}$. 
	
	We have the following estimate, see proof of [Lemma 5.5, \cite{JS}] for details. 
	$$\left|\frac{\partial^{\beta}}{\partial\xi^{\beta}}M_{s,\mu}(\xi)\right|\leq c L (1+|\mu|)^N 2^{(1+\epsilon)k|\beta|}$$ for all $|\beta|\leq N.$ Also, observe that the supremum over $s$ is outside the integral in~\eqref{Qoper} we can use Plancherel theorem to prove it when $p=2$. 
	
	Next, using integration by parts and the pointwise estimate on $\frac{\partial^{\beta}}{\partial\xi^{\beta}}M_{s,\mu}(\xi)$ as above we can show that
	\begin{equation}\label{pointwiseestimate2}
		\left|\sum_{i=1}^n\frac{\partial}{\partial x_i}\mathcal{F}^{-1}M_{R, s,\mu}(x)\right|\leq c LR^{n+1} (1+|\mu|)^N (1+2^{-(1+\epsilon)k}|Rx|)^{-N'}
	\end{equation}
	where $N'\leq N$.
	
	Choose $N'$ large enough in the inequality \eqref{pointwiseestimate2} to show that $\mathcal{F}^{-1}M_{R,s,\mu}$ verifies the conditions of vector-valued Calderon-Zygmund kernel. More precisely, for $x\neq 0,$ we have 
	$$\sup_{s\in [s_1,s_2]}\left(\int_{0}^{\infty}\left|\sum_{i=1}^n\frac{\partial}{\partial x_i}\mathcal{F}^{-1}M_{R,s,\mu})(x)\right|^2\frac{dR}{R}\right)^{\frac{1}{2}}\lesssim L (1+|\mu|)^N(2^{-(1+\epsilon)k}|x|)^{-n-1}.$$
	Using the vector-valued Calder\'{o}n-Zygmund theory we first show that $f\rightarrow \left(\int_{0}^{\infty}|Q_{N,R,s}^{k}f|^2\frac{dR}{R}\right)^{\frac{1}{2}}$ is weak type $(1,1)$ which implies the result for $1<p<2$. The estimate for  $2<p<\infty$ is obtained by duality. This yields the desired estimate \eqref{Qoper}. 
\end{proof}
From the decomposition of $U^{\varphi_0,j}_{Rs,R}(k,m)f$ (see~\eqref{decompositionofU} observe that the  
estimate~ \eqref{reduction1} and Lemma \ref{reduction2} yield the desired estimate for operator $U^{\varphi_0,j}_{Rs,R}(k,m)f$ as claimed in Proposition \ref{maxsquare2}. This takes care of first part in the expression \eqref{microdecomposition}. 

Next, we need to deal with the operators $P_{Rs,R}^{j,k}(\mu,m)$ to complete the proof of Proposition~\ref{maxsquare2}. This part follows using more or less the same arguments as used in proving Lemma~\ref{reduction2}.  With $m\leq [\nu^{-1}]+1$ we decompose $P_{Rs,R}^{j,k}(\mu,m)$ as follows. 
\begin{eqnarray*}\label{decompositionofP}
	P_{Rs,R}^{j,k}(\mu,m)f(x)&=&\sum_{l=0}^{N-1}\frac{1}{l!}s^22^{-d(1+\epsilon)l}2^{-j\epsilon l}\psi_j(k,m,l)D_{s,R}^{l,k}(\mu,m)f(x)\\
	&&+	\int_{\R}\int_{\R^n}\hat{\psi}(\mu')s_N^R(\mu,\mu')\varphi(2^{(1+\epsilon)k}\xi_{k,m}^{Rs})\hat{f}(\xi)e^{2\pi ix.\xi} e^{2\pi i2^j(1-a-s^22^{-(1+\epsilon)k}m)\mu'}d\mu'd\xi,
\end{eqnarray*}
where $\psi_j(k,m,l)=\frac{d^l\psi}{dx^l}(2^j(1-a-s^22^{-(1+\epsilon)k}m)),~ s_N^R(\mu,\mu')=r_N(2^k\xi_{k,m}^{Rs}\mu)r_N(2^js^2\xi_{k,m}^{Rs}\mu')$ and for $0\leq l\leq N-1$ the operator $D_{s,R}^{l,k}(\mu,m)$ is defined by 
\begin{equation*}\label{microterms1} D_{s,R}^{l,k}(\mu,m)f(x)=\int_{\R^n}\varphi_l(2^{(1+\epsilon)k}\xi_{k,m}^{Rs})r_N(2^k\xi_{k,m}^{Rs}\mu)e^{2\pi ix.\xi}\hat{f}(\xi)d\xi.
\end{equation*} 
The square function corresponding to each of these operators can be dealt with in a similar way as we did for $X_{N,R,s}^k$ in Lemma \ref{reduction2} with the bounds being uniform in $s\in[s_1,s_2]$. We leave the details here. This completes the proof of Proposition~\ref{maxsquare2}. \qed

Putting the estimates~\eqref{mainestiforII}, \eqref{gjj21} and Theorem~\ref{maxsquare} we get desired $L^p$ estimates for $\G^{\alpha}_j$ for $j\geq 2.$
\noindent
\subsection*{Boundedness of $\G^{\alpha}_0$:}
This part is dealt with by making minor modifications to the arguments used in the previous case. We need to carry put a similar decomposition with respect to the other variable $\eta$ in order to deal with the square function $\G^{\alpha}_0$. We leave the details to avoid repetion as this part can be completed following the steps in the previous section along with the correspondingideas from~\cite{JS}
\begin{remark}
The proof of Theorems~\ref{dim1} and~ \ref{maintheorem:sqr} can be simplified to a great extent for the particular case of $p_1=p_2=2$ as compared to other values of exponents. The details are left to the reader.
\end{remark}
\section{Proof of Theorem~\ref{dim1}}\label{sec:dim1}
We follow the proof of Theorem~\ref{maintheorem:sqr} in Section~\ref{sec:pqr} and arrive at equation~\eqref{decom4}. Then we estimate the square functions corresponding to terms $I_{j,R}$ and $II_{j,R}$. Consider the second part $II_{j,R}$ as we did in the proof of Theorem~\ref{maintheorem:sqr} in Section~\ref{sec:pqr}. Invoking \cite[Theorem $6.1$]{JS} we know that  
$$\left\|\sup_{R>0}\left(\int_{0}^{\sqrt{2^{-j+1}}}|\tilde{S}_{j,\beta}^{R,t}f(\cdot)t^{2\delta+1}|^2dt\right)^{\frac{1}{2}}\right\|_{p_1}\lesssim2^{j(\frac{1}{4}-\tilde{\alpha})}\|f\|_{p_1},$$
for $\beta>\frac{1}{2}$, when $2\leq p_1<\infty$ and for $\beta>\frac{1}{p_1}$, when $1<p_1<2$. Here $\tilde{\alpha}=\min\{\alpha,\delta+\frac{1}{2}\}$.

The remaining term is estimated using the boundedness of Bochner-Riesz square function~\cite{S,KanSun}. We have 
$$\left\|\left[\int_0^\infty\left(\int_{0}^{\sqrt{2^{-j+1}}}|A_{Rt}^{\delta}g(x)|^2dt\right)\frac{dR}{R}\right]^{\frac{1}{2}}\right\|_{p_2}\leq 2^{\frac{5}{4}(-j+1)}\|g\|_{p_2}$$
for $\delta>-\frac{1}{2}$ when $2\leq p_2<\infty$ and for $\delta>\frac{1}{p_2}-1$ when $1<p_2<2$.

From these estimates we get that 
\begin{eqnarray}\label{gjn=1} \left\|\left(\int_0^\infty|II_{j,R}|^2\frac{dR}{R}\right)^{\frac{1}{2}}\right\|_{p}\lesssim2^{j(\frac{1}{4}-\tilde{\alpha})+\frac{5}{4}(-j+1)}\|f\|_{p_2}\|g\|_{p_2},
	\end{eqnarray}
holds for for the claimed range of $p_1,p_2,p$ and $\alpha$ in Theorem~\ref{dim1}.

Next, consider
\begin{eqnarray*} 
	&&\left\|\left(\int_0^\infty |I_{j,R}|^2\frac{dR}{R}\right)^{\frac{1}{2}}\right\|_{p} \\
	&\leq&\nonumber 2^{-\frac{j}{4}} \left\|\left[\int_0^\infty\left(\int_{0}^{\sqrt{2^{-j+1}}}|{S}_{j,\beta}^{R,t}f(\cdot)t^{2\delta+1}|^2dt\right)\frac{dR}{R}\right]^{\frac{1}{2}}\right\|_{p_1} \left\|\sup_{R>0}\left(\frac{1}{R_j}\int_0^{R_j}|B_t^{\delta}g(x)|^2dt\right)^{\frac{1}{2}}\right\|_{p_2}.
\end{eqnarray*}
Recall that for the maximal function term in the above we have 
$$\left\|\sup_{R>0}\left(\frac{1}{R_j}\int_0^{R_j}|B_t^{\delta}g(x)|^2dt\right)^{\frac{1}{2}}\right\|_{p_2}\lesssim \|g\|_{p_2},$$
for $\delta>-\frac{1}{2}$ when $2\leq p_2<\infty$ and for $\delta>\frac{1}{p_2}-1$ when $1<p_2<2$.

Therefore, in order to complete the proof of Theorem~\ref{dim1} we need to prove the following estimate. 
\begin{eqnarray}\label{reducedoperatord1} \left\|\left[\int_0^\infty\left(\int_{0}^{\sqrt{2^{-j+1}}}|{S}_{j,\beta}^{R,t}f(\cdot)t^{2\delta+1}|^2dt\right)\frac{dR}{R}\right]^{\frac{1}{2}}\right\|_{p_1}\lesssim 2^{-(j-1)(\delta+\frac{1}{2})}\|f\|_{p_1}
\end{eqnarray}
when $\beta>\frac{1}{2}$ for $2\leq {p_1}<\infty$ and $\beta>\frac{1}{p_1}$ for $1<{p_1}<2$.

Let us rewrite the multiplier in the following way. 
\begin{eqnarray*}
	&&\psi\left(2^j\left(1-\frac{|\xi|^2}{R^2}\right)\right)\left(1-\frac{|\xi|^2}{R^2}-t^2\right)_+^{\beta-1}\frac{|\xi|^2}{R^2}\\ &=& (1-t^2)^{\beta}\psi\left(2^j\left(1-\frac{|\xi|^2}{R^2}\right)\right)\left(1-\frac{|\xi|^2}{R^2(1-t^2)}\right)_+^{\beta-1} \frac{|\xi|^2}{R^2(1-t^2)}
\end{eqnarray*}
Denote $${S'}_{j,\beta}^{R,s}f(x)=\int_{\R}\psi\left(2^j\left(1-\frac{|\xi|^2}{R^2}\right)\right)\left(1-\frac{|\xi|^2}{s^2R^2}\right)_+^{\beta-1}\frac{|\xi|^2}{s^2R^2}\hat{f}(\xi)e^{2\pi ix.\xi}d\xi.$$
Therefore, we get that 
\begin{eqnarray*}
	\int_{0}^{\sqrt{2^{-j+1}}}|{S}_{j,\beta}^{R,t}f(\cdot)t^{2\delta+1}|^2dt
	&=&\int_{0}^{\sqrt{2^{-j+1}}}|{S'}_{j,\beta}^{R,1-t^2}f(\cdot)|^2t^{4\delta+2}(1-t^2)^{2\beta}dt \\
	&\lesssim & 2^{-j(2\delta+1)} \int_{0}^{\sqrt{2^{-j+1}}}|{S'}_{j,\beta}^{R,1-t^2}f(\cdot)|^2dt.
\end{eqnarray*}
Note that 
$${S'}_{j,\beta}^{R,s}f(x)=B^{\psi}_{2^{-j},R}B_{Rs}^{\beta-1}f(x),$$
where
$$B^{\psi}_{2^{-j},R}(f)(x)=\int_{\R}\psi\left(2^j\left(1-\frac{|\xi|^2}{R^2}\right)\right)\hat{f}(\xi)e^{2\pi ix.\xi} d\xi$$
and 
$$B^{\beta}_{Rs}(f)(x)=\int_{\R}\left(1-\frac{|\xi|^2}{R^2s^2}\right)_+^{\beta}\frac{|\xi|^2}{R^2s^2}\hat{f}(\xi)e^{2\pi ix.\xi} d\xi.$$
Since $\psi\left(2^j\left(1-\frac{|\xi|^2}{R^2}\right)\right)$ is a compactly supported smooth function, an integration by parts argument gives us that $|B^{\psi}_{2^{-j},R}f(x)|\lesssim Mf(x).$ Consequently, we get that 
$$\left\|\left(\int_{0}^{\sqrt{2^{-j+1}}}\int_0^\infty|B^{\psi}_{2^{-j},R}B_{R(1-t^2)}^{\beta-1}f(\cdot)|^2\frac{dR}{R}dt\right)^{\frac{1}{2}}\right\|_{p_1}\lesssim \left\|\left(\int_{0}^{\sqrt{2^{-j+1}}}\int_0^\infty|M B^{\beta-1}_{R(1-t^2)}f(\cdot)|^2\frac{dR}{R}dt\right)^{\frac{1}{2}}\right\|_{p_1}.$$
From here the desired result follows using a change of variables $R(1-t^2)\to R$, vector-valued boundedness of the Hardy-Littlewood maximal operator $M$ and finally the boundedness of Bochner-Riesz square function. 

This takes care of $\G^{\alpha}_j$ for $j\geq 2$. We carry out the the same analysis for the square function $\G^{\alpha}_j$ as in the proof of Theorem~\ref{maintheorem:sqr} and use the one-dimensional results for Bochner-Riesz square function to conclude the proof of Theorem~\ref{dim1}. Since this part does not require any modification in the proof given in the previous section, we leave the details to the reader. \qed
\section{Analytic interpolation for $\mathcal G^{\alpha}$ and proof of Theorem~\ref{thm:inter}}\label{sec:inter} 
In this section we obtain new results for $\G^{\alpha}$ by using multilinear complex interpolation. In order to do so we will require $L^p$ estimates for $\G^{\alpha}$ for $\alpha>n-\frac{1}{2}$, which is the subject of  Theorem~\ref{thm:critical}, see Appendix for a proof. We assume it for a moment and complete the proof of Theorem~\ref{thm:inter}. 

Let $z=\alpha+i\tau\in \C$ and consider the linearised operator with complex order 
	\begin{equation*}
		\mathcal T^z_b(f,g)(x)=\int_0^\infty\int_{\R^n\times\R^n}\left(1-\frac{|\xi|^2+|\eta|^2}{R^2}\right)^{z}_{+}\frac{|\xi|^2+|\eta|^2}{R^2}\hat{f}(\xi)\hat{g}(\eta)e^{2\pi ix.(\xi+\eta)}d\xi d\eta b(x,R)\frac{dR}{R},
	\end{equation*}
where $b(x,R)\in L^2((0,\infty),dR/R)$ with $\int_0^\infty|b(x,R)|^2\frac{dR}{R}\leq 1.$ 

We need to track down the bounds for the operator norm $\|\G^z\|_{L^{p_1}\times L^{p_2}\rightarrow L^p}$ in terms of the imaginary part of $z$ in Theorems~\ref{maintheorem:sqr} and \ref{thm:critical}. This is possible as we have precise formula for constants in terms of the parameter $z$ in the decomposition~\eqref{decom3} of the bilinear multiplier symbol $\left(1-\frac{|\xi|^2+|\eta|^2}{R^2}\right)^{z}_{+}$ which is given by  
$$c_z=C_{\beta,\delta+\iota\tau}=\frac{\Gamma(\beta+\delta+1+\iota\tau)}{\Gamma(\beta)\Gamma(\delta+1+\iota\tau)},$$ where $\alpha=\beta+\delta$ with $\beta>\frac{1}{2}$ and $\delta>-\frac{1}{2}$, see equation~\eqref{constant:stein} and \cite[page $279$]{SW} for details. Using estimates of gamma function from [\cite{Grafakosclassical}, page no. 422-423] we have 
	\begin{eqnarray*}
		|\Gamma(\beta+\delta+1+\iota\tau)|\leq |\Gamma(\beta+\delta+1)|\quad\text{and}\quad \frac{1}{|\Gamma(\delta+1+\iota\tau)|}\leq \frac{e^{C(\delta)|\tau|^{2}}}{|\Gamma(\delta+1)|},
	\end{eqnarray*}
	where $C(\delta)=\max\{(1+\delta)^{-2},(1+\delta)^{-1}\}$.
	 
Therefore, $|c_z|$ increases atmost by a constant multiplle of $e^{C|\tau|^2}$ where  $C$ is a fixed constant.  

Further, we need to employ Calder\'{o}n-Zygmund method, in which we need to use estimates on the kernel of $\G^{\alpha}$. Again for this part, we have the following estimates of Bessel functions, for $Re(z)>-\frac{1}{2}$ and $|x|\geq1$
	$$|J_z(|x|)|\leq C_{Re(z)}e^{a_1|Im(z)|^2}|x|^{-\frac{1}{2}}$$
	where $a_1=\max\{(Re(z)+\frac{1}{2})^{-2},(Re(z)+\frac{1}{2})^{-1}+\pi/2\}$ and for $Re(z)>-\frac{1}{2}$ and $0<|x|\leq1$
	$$|J_z(|x|)|\leq C'_{Re(z)}e^{a_2|Im(z)|^2}|x|^{Re(z)}$$
	where $a_2=\max\{(Re(z)+\frac{1}{2})^{-2},(Re(z)+\frac{1}{2})^{-1}\}$.
Note that, for the kernel estimates we have constants with growth atmost by $e^{C_2|\tau|^2}$. 

With this discussion on bounds for various constants as above, an inspection of proof of Theorem~\ref{maintheorem:sqr} yields that 	$\|\G^z\|_{L^{2}\times L^{2}\rightarrow L^1}$ increases by a constant multiple of $e^{C|\tau|^2}$, where $z=\alpha+i\tau$ with $\alpha>0$. Consequently, the bilinear Calder\'{o}n-Zygmund theory~(see \cite{GT} for details) yields that  $\|\G^z\|_{L^{p_1}\times L^{p_2}\rightarrow L^{p,\infty}}$ increases atmost as $e^{c|\tau|^2}$ for some $c>0$ where $1\leq p_1,p_2<\infty$  and $\alpha>n-\frac{1}{2}$. Finally, the real interpolation between these two estimates gives us similar bounds for the quantity $\|\G^z\|_{L^{p_1}\times L^{p_2}\rightarrow L^{p}}$ for all $1<p_1,p_2<\infty$ and $\alpha>n-\frac{1}{2}$. These estimates  ensure the admissible growth while we try to use bilinear analytic interpolation theorem (see Theorem 7.2.9,  \cite{Grafakosmodern}) for $\G^z$. 

Denote $V=\{z\in\C:0<Re(z)<1\}$. For a fixed $\delta>0$ and $z\in \overline{V}$  consider the family of bilinear operators $U^z(f,g)=\mathcal T_b^{(n-\frac{1}{2})z+\delta}(f,g)$, where $f$ and $g$ are Schwartz class functions.
Consider 
	\begin{eqnarray*}
		&&|U^z(f,g)(x)|\\
		&=& \left|\int_0^\infty\int_{\R^n\times\R^n}\left(1-\frac{|\xi|^2+|\eta|^2}{R^2}\right)^{(n-\frac{1}{2})z+\delta}_{+}\frac{|\xi|^2+|\eta|^2}{R^2}\hat{f}(\xi)\hat{g}(\eta)e^{2\pi ix.(\xi+\eta)}d\xi d\eta~ b(x,R)\frac{dR}{R}\right|\\
		&\leq& \int_{\R^n\times\R^n}\left(\int_{\sqrt{|\xi|^2+|\eta|^2}}^\infty \left|\left(1-\frac{|\xi|^2+|\eta|^2}{R^2}\right)^{(n-\frac{1}{2})z+\delta}_{+}\frac{|\xi|^2+|\eta|^2}{R^2}\right|^2\frac{dR}{R}\right)^{\frac{1}{2}} |\hat{f}(\xi)\hat{g}(\eta)|d\xi d\eta.
	\end{eqnarray*}
Here in the above we have used Cauchy-Schwarz inequality in $R$ variable and the fact that $\|b\|_{L^2}\leq 1$. Making a change of variable $\frac{\sqrt{|\xi|^2+|\eta|^2}}{R}\to s,$ we get
	\begin{eqnarray*}
		\int_{\sqrt{|\xi|^2+|\eta|^2}}^\infty \left|\left(1-\frac{|\xi|^2+|\eta|^2}{R^2}\right)^{(n-\frac{1}{2})z+\delta}_{+}\frac{|\xi|^2+|\eta|^2}{R^2}\right|^2\frac{dR}{R}
		&=& \int_0^1 (1-s^2)^{Re((2n-1)z+2\delta)}s^3~ds.
	\end{eqnarray*}
	Since $Re((n-\frac{1}{2})z+\delta)>0$, the integral in the above expression is finite. Therefore, we get that 
	$$|U^z(f,g)(x)|\lesssim \|\hat{f}\|_{L^1}\|\hat{g}\|_{L^1}.$$
	From this estimate we get that for the family $\{U^z\}$  and $0\leq a<\pi$, we have
	$$\sup_{z\in \overline{V}}\frac{\log|(U^z(f,g)(x)|}{e^{a|Im(z)|}}<\infty.$$ 
	
	From this estimate it is clear that the map $z\to U^z(f,g)$ is analytic in $V$, continuous and bounded on the closure $\overline{V}$ and the family $\{U^{z}\}_{z\in\overline{V}}$ is of admissible growth, see [page 513, \cite{Grafakosmodern}]. 
	
	Note that when $Re(z)=0$, the operator $U^z$ maps $L^2(\R^n)\times L^2(\R^n)\to L^1(\R^n)$ with norm bounded by  $K_0(0+i\tau)=C_{n,\delta}e^{C|\tau|^2}$ for some $C, C_{n,\delta}>0$. When $Re(z)=1$, we have that $U^z$ maps $L^{p_{1}}(\R^n)\times L^{p_{2}}(\R^n)\rightarrow L^{p}(\R^n)$ with a bound $K_1(1+i\tau)=C_{n,\delta}e^{c |\tau|^2}$ for some $c>0$. 
	
	Given a triplet $(q_1,q_2,q)$ of exponents satisfying the H\"{o}lder relation we choose another triplet  $(p_1, p_2,p)$ satisfying the H\"{o}lder relation such that  $\frac{1}{q_{i}}=\frac{1-\theta}{2}+\frac{\theta}{p_{i}}, i=1,2$ where $0\leq \theta \leq 1$. Now applying the bilinear analytic interpolation 
	\cite[Theorem $7.2.9$]{Grafakosmodern} we get that 
	$\mathcal T_b^{\alpha(q_{1},q_{2})}$ is bounded from $L^{q_{1}}(\R^n)\times L^{q_{2}}(\R^n)\to L^{q}(\R^n)$
	where $0<\alpha(q_{1},q_{2})<n-\frac{1}{2}$ and it is determined by the relation $(n-\frac{1}{2})\theta+\delta=\alpha(q_{1},q_{2})$ for arbitrarily small $\delta>0$. This in turn implies that $\G^{\alpha(q_{1},q_{2})}$ is bounded from $L^{q_{1}}(\R^n)\times L^{q_{2}}(\R^n)\to L^{q}(\R^n)$. 

We would like to remark that the discussion above gives us new results for $\G^{\alpha}$ with $\alpha$ below the critical index when either of the exponents $q_1$ or $q_2$ is less than $2$.  In particular, if we consider the case of $1<q_1=q_2<2$, we can write down the explicit range of the index $\alpha$. Write $q_1=q_2=q$. Note that we can choose $p_1=p_2$ arbitrarily close to $1$, say $1+\epsilon$ where $\epsilon$ is very small, then from the convex combination relation we get that $ \theta=\frac{1+\epsilon}{1-\epsilon}(\frac{2}{q}-1)$. Hence we get that the range of $\alpha$ is given by $\alpha=\alpha(q_1,q_2)>(2n-1)(\frac{1}{q}-\frac{1}{2})$. This proves Theorem~\ref{thm:inter}. 
\qed

\section{Necessary conditions: Proof of Proposition~\ref{prop:neccond}}\label{sec:nec}
The following result will be required in our analysis. 
\begin{lemma}\label{improving}
	Let $\alpha, p_1, p_2, p$ be a given parameters for which the bilinear square function $\mathcal G^{\alpha}$ maps $L^{p_1}(\R^n) \times L^{p_2}(\R^n)$ to $L^p(\R^n)$. Then $\mathcal G^{\alpha+1}$ is also bounded from $L^{p_1}(\R^n) \times L^{p_2}(\R^n)$ to $L^p(\R^n)$ with norm bounded by a constant multiple of $\|\mathcal G^{\alpha}\|_{L^{p_1}\times L^{p_2}\rightarrow L^p}$. 
\end{lemma}
\begin{proof}
	For $\rho>\alpha+\frac{1}{2}$ we have the following identity for the Bochner-Riesz multiplier symbol from~[\cite{St}, page $105$]
	\begin{eqnarray*}
		\left(1-|\xi|^2\right)^{\rho}_{+}=C_{\rho, \alpha}\int_0^1(1-t^2)^{\rho-\alpha-1}t^{2\alpha+1}\left(1-\frac{|\xi|^2}{t^2}\right)^{\alpha}_{+}dt
	\end{eqnarray*}
	Using this identity for $\rho=\alpha+1>\alpha+\frac{1}{2}$ we get that  
	\begin{eqnarray*}
		\mathcal K^{\alpha+1}_R\ast (f,g)(x)&=& C_{\alpha}	\int_{\R^n\times\R^n}\frac{|\xi|^2+|\eta|^2}{R^2}\left(1-\frac{|\xi|^2+|\eta|^2}{R^2}\right)^{\alpha+1}_{+}\hat{f}(\xi)\hat{g}(\eta)e^{2\pi ix.(\xi+\eta)}d\xi d\eta\\
		&=& C_{\alpha}	\int_0^1t^{2\alpha+1}\int_{\R^n\times\R^n}\frac{|\xi|^2+|\eta|^2}{R^2}\left(1-\frac{|\xi|^2+|\eta|^2}{t^2R^2}\right)^{ \alpha}_{+}\hat{f}(\xi)\hat{g}(\eta)e^{2\pi ix.(\xi+\eta)}d\xi d\eta dt\\
		&=& C_{ \alpha}	\int_0^1t^{2\alpha+3} \mathcal K^{\alpha}_{tR}\ast (f,g)(x)dt
	\end{eqnarray*}
	This relation along with elementary arguments gives us the pointwise relation. 
	\begin{eqnarray*}
		\mathcal G^{\alpha+1}(f,g)(x)\lesssim \mathcal G^{\alpha}(f,g)(x).
	\end{eqnarray*}
\end{proof}
\subsection*{Proof of Proposition~\ref{prop:neccond} part $(1)$:}
	Let $\psi$ be a function such that $\widehat{\psi}$ is compactly supported and smooth with  $\widehat{\psi}(\xi)\equiv 1$ for $|\xi|\leq 2$. Recall that $\mathcal{K}^{\alpha}_R(y)= c_{\alpha}R^{2n}\left(\frac{J_{\alpha+n}(2\pi R|y|)}{(2\pi R|y|)^{\alpha+n}}-\frac{J_{\alpha+n+1}(2\pi R|y|)}{(2\pi R|y|)^{\alpha+n+1}}\right), y\in \R^{2n}$. Using the asymptotic properties of the Bessel functions (see \cite[page $432$]{Grafakosclassical}), for large $|(x,x)|$ we get that  $$\mathcal{K}^{\alpha}_R*(\psi,\psi)(x) = c_1 \frac{\cos \left(2\pi R|(x,x)|+\frac{\pi}{2}(\alpha+n+\frac{1}{2})\right)}{|(x,x)|^{-(\alpha+n+\frac{1}{2})}}+O\left(\frac{1}{|(x,x)|^{-(\alpha+n+3/2)}}\right).$$ 
	
	Let $M\in \N$ and define $A_M(R)=\{x\in \R^n: \left||x|- \frac{M-\frac{\pi}{2}(\alpha+n+\frac{1}{2})}{\sqrt{2}R}\right|<\frac{\delta}{2\pi R}\}$ for $\delta<<1/4$.
	From the boundedness of $\G^{\alpha}$ we have $$\left\|\left(\int_{1}^2 \left|\mathcal{K}^{\alpha}_R*(\psi,\psi)(x)\right|^2 dR\right)^{\frac{1}{2}}\right\|_p\lesssim \|\psi\|_{p_1}\|\psi\|_{p_2}<\infty.$$ 
	This implies that \begin{equation}\label{necc}\limsup_{M\rightarrow\infty}\int_1^2\int_{x\in A_M(R)}|x|^{-p(\alpha+n+\frac{1}{2})}dx dR <\infty\end{equation}  
	Since $R$ is localised between $[1,2]$, the condition \eqref{necc} implies that $\alpha>n\left(\frac{1}{p}-1\right)-\frac{1}{2}$.
	
	Next, we need to show that we must have $\alpha>-\frac{1}{2}.$ 
	Consider $$
	\int_1^2|\mathcal B^{\alpha}_R(f,g)(x)|^2dR=\int_1^2\left|\int_{\R^n\times\R^n}\left(1-\frac{|\xi|^2+|\eta|^2}{R^2}\right)^{\alpha}_+\hat{f}(\xi)\hat{g}(\eta)e^{2\pi ix.(\xi+\eta)}d\xi d\eta\right|^2 dR $$
	Adding and subtracting terms $\left(1-\frac{|\xi|^2+|\eta|^2}{R^2}\right)^{\alpha+j}_+$ from $j=1$ to $N$ with $\alpha+N> n-\frac{1}{2}$ we get that 
	  \begin{equation}\label{SS}
		\left(\int_1^2|\mathcal B^{\alpha}_R(f,g)(x)|^2dR\right)^{\frac{1}{2}}\leq \sum_{j=0}^{N-1}\mathcal{G}^{\alpha+j}(f,g)(x)+\left(\int_1^2|\mathcal B^{\alpha+N}_R(f,g)(x)|^2dR\right)^{\frac{1}{2}}	\end{equation}
	Lemma~\ref{improving} and the estimates on $\K^{\alpha}_R$ for  $\alpha>n-\frac{1}{2}$ yield that $(f,g)\rightarrow \left(\int_1^2|\mathcal B^{\alpha}_R(f,g)(x)|^2dR\right)^{\frac{1}{2}}$ is bounded for the triplet $(p_1,p_2,p)$ under consideration.
	Using linearization argument as earlier we see that for any $b\in L^2([1,2], dR)$ the bilinear operator 
	$$\mathcal T_b^{\alpha}(f,g)(x)=\int_1^2 b(R)\left(\int_{\R^n\times\R^n}\left(1-\frac{|\xi|^2+|\eta|^2}{R^2}\right)^{\alpha}_+\hat{f}(\xi)\hat{g}(\eta)e^{2\pi ix.(\xi+\eta)}d\xi d\eta\right) dR$$
	is bounded from $L^{p_1}(\R^n)\times L^{p_2}(\R^n)$ into $L^p(\R^n)$. Note that if $\widehat{f}$ and ${\widehat{g}}$ are compactly supported smooth functions, then $\mathcal T_b^{\alpha}(f,g)$ makes sense for $\alpha\geq -\frac{1}{2}$.
	
	Let $u,v\in \mathbb S^{n-1}$ and $\rho>0$.  Using translation and dilation arguments for bilinear multiplier operators we see that the operator 
	\begin{equation*}
		\mathcal T_{b,B_{u,v,\rho}}^{\alpha}(f,g)(x)=\int_1^2 b(R)\left(\int_{\R^n\times\R^n}\left(1-\frac{|\xi-\rho u|^2 + |\eta-\rho v|^2}{2R^2\rho^2}\right)^{\alpha}_+\hat{f}(\xi)\hat{g}(\eta)e^{2\pi ix.(\xi+\eta)}d\xi d\eta\right)dR.
	\end{equation*}
	is bounded form $L^{p_1}(\R^n)\times L^{p_2}(\R^n)$ into $L^p(\R^n)$ with operator norm uniformly bounded in $u,v\in \mathbb S^{n-1}$ and $\rho>0$. 
	
	Note that the multiplier $\left(1-\frac{|\xi-\rho u|^2 + |\eta-\rho v|^2}{2R^2\rho^2}\right)^{\alpha}_+$ converges to $(1-\frac{1}{R^2})^{\alpha}_+$ as $\rho\rightarrow \infty$ for any fixed $\xi,\eta$. The boundedness of $\mathcal T_{b,B_{u,v,\rho}}^{\alpha}(f,g)$ is bounded on the triplet $(p_1,p_2,p)$ and Fatou's lemma imply that $$\left|\int_1^2 \left(1-\frac{1}{R^2}\right)^{\alpha}_+ b(R)dR\right|\lesssim \|b\|_2,$$
	for all $b\in L^2([1,2])$. The Riesz representation theorem for $L^2([1,2])$ implies that $\alpha>-\frac{1}{2}$. This completes the proof of part $(1)$ of Proposition~\ref{prop:neccond}. 
\qed
\subsection*{Proof of Proposition~\ref{prop:neccond} part $(2)$:}
		We already have that $\alpha>-\frac{1}{2}$. Therefore, we need to prove the remaining two conditions. Consider the linearised local square function given by  
		$$\mathcal L^\alpha_b(f,g)(x)=\int_1^2 b(R)\left(\int_{\R^n\times\R^n}\left(1-\frac{|\xi|^2+|\eta|^2}{R^2}\right)^{\alpha}_{+}\frac{|\xi|^2+|\eta|^2}{R^2}\hat{f}(\xi)\hat{g}(\eta)e^{2\pi ix.(\xi+\eta)}d\xi d\eta\right) dR,$$
		where $b\in L^2([1,2])$. 
		
		For $\eta=(\eta_1,\eta_2,\dots, \eta_n)\in \R^n$ write $\eta=(\eta',\eta_n),$ where $\eta'=(\eta_1,\eta_2,\dots, \eta_{n-1})$. Let $\epsilon>0$ be a small number and let   $\psi_\epsilon(\xi,\eta)=\phi_1(\xi/\epsilon)\phi_2(\eta'/\epsilon)\phi_3((1-\eta_d)/\epsilon)$ where  $\phi_1, \phi_2$ and $\phi_3$
		are smooth functions supported in the ball $B(0,1)$. 
		
		Consider the bilinear operator 
		$$\tilde{\mathcal L}^\alpha_b(f,g)(x)=\int_1^2 b(R)\left(\int_{\R^n\times\R^n}\psi_\epsilon(\xi,\eta)\left(1-\frac{|\xi|^2+|\eta|^2}{R^2}\right)^{\alpha}_{+}\frac{|\xi|^2+|\eta|^2}{R^2}\hat{f}(\xi)\hat{g}(\eta)e^{2\pi ix.(\xi+\eta)}d\xi d\eta \right)dR.$$
		Observe that by clubbing $\phi_1(\xi/\epsilon)$ with $\hat f(\xi)$ and $\phi_2(\eta'/\epsilon)\phi_3((1-\eta_d)/\epsilon)$ with $\hat g(\eta)$, the $L^p$ boundedness of $\mathcal L^\alpha_b$ implies the corresponding $L^p$ boundedness of the operator $\tilde{\mathcal L}^\alpha_b$. Denote 
		$$\mathcal{K}_{\psi,R}^{\alpha}(\xi,\eta)=\int_{\R^n\times \R^n} \psi_\epsilon(\xi,\eta)\left(1-\frac{|\xi|^2+|\eta|^2}{R^2}\right)^{\alpha}_{+}\frac{|\xi|^2+|\eta|^2}{R^2}e^{-2\pi i(y\cdot\xi+z\cdot\eta)} d\xi d\eta. $$ 
		
		Let $\phi$ be a smooth and compactly supported function with $\phi=1$ on $B(0,2)$ and observe that 
		\begin{eqnarray}
			\left|	\int_{\R^n} \tilde{\mathcal L}^\alpha_b(f,g)(x) \check{\phi}(x) dx\right|
			&=& \nonumber \left| \int_1^2\left(\int_{\R^n\times \R^n} \mathcal{K}_{\psi,R}^\alpha(y,z) f(y)g(z)dy dz \right)b(R)dR\right| \\
			&\lesssim &\label{nec111} \|b\|_2\|f\|_{p_1}\|g\|_{p_2}.
		\end{eqnarray}
		Consider a narrow cone  $\mathcal C=\{(y,z)\in \R^n\times \R^n: \sqrt{|y|^2+|z'|^2}\leq \epsilon_0 z_n\},$ where $\epsilon_0<\epsilon$. Using the method of stationary phase we know that for $w=(y,z)\in \mathcal C$ the kernel $\mathcal{K}_{\psi,R}^\alpha$ satisfies the following estimate 
		$$\mathcal{K}_{\psi,R}^\alpha(w)=R^{2n}{e^{i2\pi R|w|}}{|Rw|^{-\frac{2n+1}{2}-\alpha}} + O(|Rw|^{-\frac{2n+1}{2}-\alpha-1}).$$ See \cite[Lemma $2.3.3$]{CSO} for details.
		Let $M\gg \epsilon_0^{-100}$ be a large number. Consider the sets $A_M=\{y: (\epsilon_0/10)M^{\frac{1}{2}}\leq |y|< (\epsilon_0/5)M^{\frac{1}{2}}\}$ and $B_M=\{z: (\epsilon_0/10)M\leq |z|\leq (\epsilon_0/5)M,\, |z'|\leq (\epsilon_0/10)|z_n| \}.$ Note that $A_M\times B_M\subset \mathcal C$. Let $f(y)=\chi_{A_M}(y)$ and $g(z)=\chi_{B_M}(z)e^{-i2\pi|z|}.$ 
		Thus,
		\begin{eqnarray*}
			&&	\left|\int_1^2 \int_{\R^n\times \R^n} \mathcal{K}_R^\alpha(y,z) f(y)g(z)dy dz b(R)dR\right|\\
			&\gtrsim &  \left|\int_1^2\int_{A_M \times B_M}R^{2n}e^{i2\pi(R|w|-|z|)}
			{|Rw|^{-\frac{2n+1}{2}-\alpha}}  dydzb(R)dR\right|\\
			&=& \left|\int_{A_M \times B_M}e^{-i2\pi|z|}|w|^{-\frac{2n+1}{2}-\alpha}\left(\int_1^2 R^{2n}e^{i2\pi R|w|}
			R^{-\frac{2n+1}{2}-\alpha} b(R) dR\right) dydz\right|
		\end{eqnarray*}
		Choose $b$ such that $b(R)R^{-\frac{2n+1}{2}-\alpha} R^{2n}=(R-1)^{-\delta}$ where $0<\delta<\frac{1}{2}$, then a change of variables argument gives us that 
		$$\int_1^2 R^{2n}e^{i2\pi R|w|}
		R^{-\frac{2n+1}{2}-\alpha} b(R) dR=e^{i2\pi|w|}\int_0^1e^{i2\pi R|w|} R^{-\delta}dR.$$ 
		Now we use the asymptotic estimates on $\int_0^1e^{i2\pi R|w|} R^{-\delta}dR$. First, split the integral into two parts, one where $2\pi R|w|<<1/4$ and the other one being its complement, then we use integration by parts argument on the part where $2\pi R|w|$ is large. This gives us that $$\int_1^2 R^{2n}e^{i2\pi R|w|}
		R^{-\frac{2n+1}{2}-\alpha} b(R) dR= c e^{2\pi i |w|}|w|^{-1+\delta} +O(|w|^{-1})$$ for $|w|\rightarrow\infty$.
		This estimate yields  
		\begin{eqnarray*}
			&&\left|\int_{A_M \times B_M}e^{-i2\pi|z|}|w|^{-\frac{2n+1}{2}-\alpha}\left(\int_1^2 R^{2n}e^{i2\pi R|w|}
			R^{-\frac{2n+1}{2}-\alpha} b(R) dR\right) dydz\right|\\
			&\gtrsim &\left|\int_{A_M \times B_M}e^{i2\pi(|w|-|z|)}|w|^{-\frac{2n+1}{2}-\alpha}|w|^{-1+\delta}dw\right|\end{eqnarray*}
		As shown in \cite[Proposition $4.6$]{JLV} we can verify that on $A_M\times B_M$ the term $||w|-|z||$ is small. This implies that
		\begin{eqnarray*}
			\left|\int_{A_M \times B_M}e^{i2\pi(|w|-|z|)}|w|^{-\frac{2n+1}{2}-\alpha}|w|^{-1+\delta}dw\right|
			& \gtrsim & M^{n/2-\alpha-3/2+\delta}			\end{eqnarray*}
		This estimate along with \eqref{nec111}  gives us  $$M^{-\alpha+\frac{n}{2}-3/2+\delta}\lesssim M^\frac{n}{2p_1} M^\frac{n}{p_2}.$$ 
		Since $M$ is arbitrarily large and $\delta$ can be as close to $\frac{1}{2}$ as we need, we get that   
		$\alpha>\frac{n}{2}-\frac{n}{2p_1}-\frac{n}{p_2}-1$. Using the symmetry between $\xi$ and $\eta$ we get the desired result. 
\section*{Appendix: Proof of Theorem~\ref{thm:critical}}\label{sec:critical}
We refer the reader to~\cite{LN} for details about sparse domination principle. 

\noindent
{\bf Sparse family of dyadic cubes:} Consider a dyadic lattice of cubes in $\R^n$. Then a family of dyadic cubes $\mathcal{S}$ is said to $\eta-$sparse, with $0<\eta<1$ if for every cube $Q\in \mathcal{S}$, there exists a measurable subset $E_Q\subset Q$ such that $|E_Q| \geq (1-\eta)|Q|$ and $\{E_Q\}_{Q\in\mathcal{S}}$ are pairwise disjoint.

For a locally integrable function $f$ and a finite cube $Q$ we use the notation $\langle f\rangle_Q$ to denote the average of $f$ over $Q$ give by 
$$\langle f\rangle_Q=\frac{1}{|Q|}\int_Q|f(x)|dx,$$
where $|Q|$ stands for the measure of the cube $Q$. We will always assume that cubes have their sides parallel to coordinate axes. Also, for $\lambda>0$ we use the notation $\lambda Q$ to denote the concentric cube with $Q$ such that $|\lambda Q|=\lambda^n |Q|.$ For a cube $Q\subseteq \R^n$ the notation $Q^2$ stands for the cartesian product $Q\times Q$. 

\noindent
{\bf Bilinear sparse operator:} Let $\mathcal{S}$ be a $\eta-$sparse family with $0<\eta<1$. Then for compactly supported bounded functions $f$ and $g$ the bilinear sparse operator associated with the family $\mathcal{S}$ is defined by 
$$S(f,g)(x)=\sum_{Q\in\mathcal{S}}\langle f\rangle_Q\langle g\rangle_Q\chi_Q(x).$$

We have the following pointwise sparse domination result for the square function. 
\begin{theorem}\label{thm:critical}
	Let $\alpha >n-\frac{1}{2}$. Then for compactly supported functions $f$ and $g$ defined on $\R^n$ there exists $\nu$-sparse families $\{\mathcal S_k\}_{k=1}^{3^n}$ such that 
	$$\G^{\alpha}(f,g)(x)\lesssim \sum_{k=1}^{3^n} S_k(f,g)(x),$$
	where $S_k$ denotes bilinear sparse operator defined as above and $0<\nu<1$ is a constant depending only on $n$. \\
\end{theorem}
\begin{remark}
	The proof of Theorem~\ref{thm:critical} is based on the sparse domination principle for vector-valued Calder\'{o}n-Zygmund operators. The proof follows without much difficulty using standard arguments. For completion we give the details here. Also, note that sparse domination gives us weighted estimates for the operator. However, we do not discuss weighted estimates in this paper as the main theme of the paper is to establish unweighted estimates for $\mathcal G^{\alpha}$. We leave the details of weighted consequences of Theorem~\ref{thm:critical} to the reader, see~\cite{LN} for details. 
\end{remark}

\noindent
{\bf Weak-type estimate at $(1,1,\frac{1}{2})$ for $\alpha>n-\frac{1}{2}$:}
Note that we can view the square function as a vector-valued bilinear operator in the following way. $$\G^{\alpha}(f,g)(x)=\left\|\frac{\g_R^{\alpha}(f,g)(x)}{\sqrt{R}}\right\|_{L^2}.$$
Here the norm $\|\cdot\|_{L^2}$ is taken with respect to $R$ over the interval $(0,\infty)$. Since the kernel $\mathcal K^{\alpha}_R$ is a radial function we write (with a little abuse of notation) $\mathcal K^{\alpha}_R(x,y)=\mathcal K^{\alpha}_R(r),$ where $r=|(x,y)|, (x,y)\in \R^n\times \R^n.$ 

Let $\alpha=n-\frac{1}{2}+\delta, \delta>0$. Sunouchi~\cite{S} proved the following estimate
\begin{equation}\label{kernel}
	\left|\frac{{\mathcal K^{\alpha}_R}(r+s)}{\sqrt{R}}-\frac{{\mathcal K^{\alpha}_R}(r)}{\sqrt{R}}\right|\lesssim \min\{R^{-\frac{1}{2}-\delta}r^{-(2n+\delta)},|s|R^{\frac{1}{2}-\delta}{r^{-(2n+\delta)}}\}.
\end{equation}
Consider 
\begin{eqnarray*}
	&&\int_{0<2s<r}\left\{\int_0^\infty\left|\frac{{\mathcal K^{\alpha}_R}(r+s)}{\sqrt{R}}-\frac{{\mathcal K^{\alpha}_R}(r)}{\sqrt{R}}\right|^2dR\right\}^{\frac{1}{2}}r^{2n-1}dr \nonumber\\
	&\leq& C\int_{0<2s<r}\left\{\int_0^{1/s}\left(\frac{sR^{\frac{1}{2}-\delta}}{r^{(2n+\delta)}}\right)^{2}+\int_{1/s}^\infty\left(\frac{1}{R^{\frac{1}{2}+\delta}r^{(2n+\delta)}}\right)^2dR\right\}^{\frac{1}{2}}r^{2n-1}dr <\infty
\end{eqnarray*}
Thus, we see that the kernel of $\mathcal G^{\alpha}$ verifies the regularity condition for bilinear Calder\'{o}n-Zygmund operators, see~\cite{GT, LN} for details on bilinear Calder\'{o}n-Zygmund theory. Since we already proved (see Theorem~\ref{maintheorem:sqr}) that $\mathcal G^{\alpha}$ is bounded at $(2,2,1)$, invoking the bilinear Calder\'{o}n-Zygmund theory from~Grafakos and Torres~\cite{GT} we get that  $\G^{\alpha}$ maps $L^1(\R^n)\times L^1(\R^n)$ to $L^{\frac{1}{2},\infty}(\R^n)$.\\
\noindent
\subsection*{Proof of Theorem~\ref{thm:critical}:} Let $h$ be a measurable function defined on $\R^n$ and  $E\subseteq \R^n$ be a measurable set. Define
$$\omega(h,E)=\sup_{x\in E}h(x)-\inf_{x\in E}h(x).$$
For $0<\eta<1$ and a cube $Q$ set
$$\omega_\eta(h,Q)=\min\{\omega(h,E) : E\subset Q~\text{with}~ |E|\geq(1-\eta)|Q|\}.$$
We use the following result from Lerner and Nazarov~\cite{LN} for sparse domination of $\mathcal G^{\alpha}$.  
\begin{theorem}\label{tool}\cite{LN}
	Let $f,g,h$ be functions such that
	\begin{enumerate}
		\item for every $\epsilon>0$ it holds 
		$$|\{x\in [-N,N]^n:|h(x)|>\epsilon\}|=o(N^n) \hspace{5mm}\text{ as }N\to\infty$$
		\item For any dyadic cube Q and $0<\eta \leq 2^{-n-2}$ there exists a  $\delta > 0$ such that 
		\begin{equation*}\label{oscillation}
			\omega_\eta(h,Q)\leq C_\eta \sum_{k=0}^\infty2^{-\delta k}\left(\frac{1}{|2^{k+1}Q|}\int_{2^{k+1}Q}|f|\right)\left(\frac{1}{|2^{k+1}Q|}\int_{2^{k+1}Q}|g|\right).
		\end{equation*}
	\end{enumerate} 
	Then there exists $\nu$-sparse families $\{\mathcal S_k\}_{k=1}^{3^n}$ such that for compactly supported bounded functions $f$ and $g$ we have
	$$|h(x)|\lesssim \sum_{k=1}^{3^n} S_k(f,g)(x).$$
	Here $\nu$ is a constant depending only on $n$.
\end{theorem}
We will show that the result above is application to the square function under consideration. 

Note that the condition $(1)$ of Theorem~\ref{tool} holds for $h=\G^{\alpha}(f,g)$ as it is weak-type at $(1,1,\frac{1}{2}).$ Therefore, we need to verify the second condition. 
For convenience write $T_R(f,g)(x)=\frac{\g_R^{\alpha}(f,g)(x)}{\sqrt{R}}$. Let $Q$  be a dyadic cube and set $\mathcal{Q}_{k}=(2^{k+1}Q)^2\setminus(2^{k}Q)^2,$ for $k\in \N$. Let $0<\lambda\leq2^{-n-2}$ and points $x,x'\in Q$. Consider 
\begin{eqnarray*}
	&&\left|\|T_R(f,g)(x)\|_{L^2}-\|T_R(f,g)(x')\|_{L^2}\right| \leq \|T_R(f,g)(x)-T_R(f,g)(x')\|_{L^2} \nonumber\\
	&&=\left\|T_R((f,g)\chi_{(2^{k_n}Q)^2})(x)-T_R((f,g)\chi_{(2^{k_n}Q)^2})(x')
	+\sum_{k\geq k_n}[T_R((f,g)\chi_{\mathcal{Q}_{k}})(x)-T_R((f,g)\chi_{\mathcal{Q}_{k}})(x')]\right\|_{L^2} \nonumber\\
	&&\leq I_1+I_2 \nonumber
\end{eqnarray*}
where 
$$I_1=\|T_R((f,g)\chi_{(2^{k_n}Q)^2})(x)\|_{L^2}+\|T_R((f,g)\chi_{(2^{k_n}Q)^2})(x')\|_{L^2}$$
and 
$$I_2=\sum_{k\geq k_n}\int_{\mathcal{Q}_{k}}\|\K^{\alpha}_R\left(|(x,x)-(y_1,y_2)|\right)-\K^{\alpha}_R\left(|(x',x')-(y_1,y_2|\right)\|_{L^2}|f(y_1)||g(y_2)|dy_1dy_2.$$
Note that $k_n\in \N$ is a dimensional constant which will be chosen suitably at a later stage.  

Let us first estimate the term $I_2$. Set $r=|(x',x')-(y_1,y_2)|$, since $x,x'\in Q$ and $(y_1,y_2)\in \mathcal{Q}_{k}$, we have that 
$|(x,x)-(y_1,y_2)|=r+s$ where $s\in(-\sqrt{2}|x-x'|,\sqrt{2}|x-x'|)$. Therefore,
$$\|\K^{\alpha}_R\left(|(x,x)-(y_1,y_2)|\right)-\K^{\alpha}_R\left(|(x',x')-(y_1,y_2|\right)\|^2_{L^2}=\|\K^{\alpha}_R\left(r+s\right)-\K^{\alpha}_R\left(r\right)\|^2_{L^2}.$$
We can estimate this quantity using the bounds for $\K^{\alpha}_R$ in ~\eqref{kernel}. Note that in order to use ~\eqref{kernel} we need to make sure that $r>2s$ which is possible since $x,x'\in Q$ and $(y_1,y_2)\in \mathcal{Q}_{k}$ with a choice of $k\geq k_n$. Consider 
\begin{eqnarray*}
	\left\|\K^{\alpha}_R(r+s)-\K^{\alpha}_R(r)\right\|^2_{L^2}
	&=&\int_0^{|s|^{-1}}|\K^{\alpha}_R(r+s)-\K^{\alpha}_R(r)|^2dR+\int_{|s|^{-1}}^{\infty}|\K^{\alpha}_R(r+s)-\K^{\alpha}_R(r)|^2dR\\
	&=&\int_0^{|s|^{-1}}|s|^2R^{1-2\delta}r^{-(4n+2\delta)}dR+\int_{|s|^{-1}}^{\infty}R^{-(1+2\delta)}r^{-(4n+2\delta)}dR\\
	&\approx & \frac{|s|^{2\delta}}{r^{4n+2\delta}}
\end{eqnarray*}
Let $l(Q)$ denote the sidelength of the cube $Q$. Since   $r=|(x',x')-(y_1,y_2)|\approx 2^kl(Q)$ and $s\leq\sqrt{2}|x-x'|\lesssim l(Q)$ we get that 
$$\frac{|s|^{2\delta}}{r^{4n+2\delta}}\lesssim \frac{l(Q)^{2\delta}}{2^{2k(2n+\delta)}l(Q)^{4n+2\delta}}=\left(\frac{1}{2^{k(2n+\delta)}|Q|^{2}}\right)^2.$$
Therefore, we have 
\begin{eqnarray}
	I_2&\lesssim& \sum_{k\geq k_n}\int_{(2^{k+1}Q)^2\setminus(2^{k}Q)^2}\frac{1}{2^{k(2n+\delta)}|Q|^{2}} |f(y_1)||g(y_2)|dy_1dy_2 \nonumber \\
	&\lesssim& \sum_{k\geq k_n} 2^{-k\delta} \left(\frac{1}{|2^{k+1}Q|}\int_{2^{k+1}Q}|f(y)|dy\right)\left(\frac{1}{|2^{k+1}Q|}\int_{2^{k+1}Q}|g(y)|dy\right) \nonumber
\end{eqnarray}

Next, we estimate the quantity $I_1$. This follows from the weak-type boundedness of $\G^{\alpha}$ at $(1,1,\frac{1}{2})$. For $\beta>0$ consider the set  
$$E^*=\{z\in Q:\|T_R((f,g)\chi_{(2^{k_n}Q)^2})(x)\|_{L^2}>\beta\}.$$
Using the weak-type boundedness of $\G^{\alpha}$ at $(1,1,\frac{1}{2})$ we get 
$$|E^*| \leq \left[\frac{\Vert G^{\alpha}\Vert_{L^{1}\times L^{1}\rightarrow L^{\frac{1}{2},\infty}}}{\beta}\left(\int_{2^{k_n}Q}|f(y)|dy\right)\left(\int_{2^{k_n}Q}|g(y)|dy\right)\right]^{\frac{1}{2}}$$
We can choose $\beta=\Vert G^{\alpha}\Vert_{L^{1}\times L^{1}\rightarrow L^{\frac{1}{2},\infty}}2^{2nk_{n}}\lambda^{-2}\left(\frac{1}{|2^{k_n}Q|}\int_{2^{k_n}Q}|f(y)|dy\right)\left(\frac{1}{|2^{k_n}Q|}\int_{2^{k_n}Q}|g(y)|dy\right)$. This implies that $|E^*|\leq\lambda|Q|$. 
Take $E=Q\setminus E^*$ and observe that for $x\in E$ we have 
$$\|T_R((f,g)\chi_{(2^{k_n}Q)^2})(x)\|_{L^2(0,\infty)}\lesssim_\lambda \left(\frac{1}{|2^{k_n}Q|}\int_{2^{k_n}Q}|f(y)|dy\right)\left(\frac{1}{|2^{k_n}Q|}\int_{2^{k_n}Q}|g(y)|dy\right).$$
Moreover, we have that 
$$|E|\geq|Q|-|E^*|\geq(1-\lambda)|Q|.$$
Putting these estimates together we get that with our choice of $E$ for every $x,x'\in E$ the following holds.  
\begin{eqnarray*}
	&&\left|\|T_R(f,g)(x)\|_{L^2}-\|T_R(f,g)(x')\|_{L^2}\right| \\
	&\lesssim& C_\lambda \sum_{k=0}^\infty 2^{-k\delta} \left(\frac{1}{|2^{k+1}Q|}\int_{2^{k+1}Q}|f(y)|dy\right)\left(\frac{1}{|2^{k+1}Q|}\int_{2^{k+1}Q}|g(y)|dy\right) \nonumber
\end{eqnarray*}
This proves condition $(2)$ of Theorem~\ref{tool} and hence the proof of Theorem~\ref{thm:critical} is done.  

\end{document}